\documentclass[a4paper,11pt]{amsart}

\usepackage[left=2.5cm,right=2.5cm,top=4cm,bottom=4cm]{geometry}
 
\usepackage{amsfonts}

\usepackage{amsmath}
\usepackage{amssymb}
\usepackage{amsthm}
\usepackage[final]{graphicx}
\usepackage[all]{xy}
\usepackage{float}
\usepackage{fancybox}
\usepackage{color}
\usepackage{hyperref}
\usepackage[sort&compress,comma,square,numbers]{natbib}
\usepackage{multicol}
\usepackage{enumerate}
\usepackage[version=3]{mhchem}
\usepackage{xfrac}

\newcommand{\IS}{\mathcal{I}}

\newcommand{\VS}{\mathcal{V}}
\newcommand{\Vp}{\mathcal{V}}

\newcommand{\VR}{\mathbb{V}_{\mathbb{R}}}
\newcommand{\VC}{\mathbb{V}_{\mathbb{C}}}
\newcommand{\Rp}{\mathbb{R}_{>0}}
\newcommand{\R}{\mathbb{R}}
\newcommand{\Q}{\mathbb{Q}}
\newcommand{\C}{\mathbb{C}}

\newcommand{\Sen }{\mathcal{S}}

\newcommand{\Sreg}{\mathcal{C}_{reg}}

\newcommand{\Sing}{\mathrm{Sing}}
\newcommand{\Snd}{\mathcal{C}_{nd}}

\newcommand{\Deg}{\mathrm{Deg}}
\newcommand{\NDeg}{\mathrm{Ndeg}}
\newcommand{\Vreg}{\mathcal{V}_{reg}}

\newcommand{\Vnd}{\mathcal{V}_{nd}}

\newcommand{\SndS}{\mathcal{C}_{nd}(g)_S}
\newcommand{\VndS}{\mathcal{V}_{nd}(g)_S}

\DeclareMathOperator*{\diag}{diag}
\DeclareMathOperator*{\rank}{rank}

\theoremstyle{theorem}
\newtheorem{Thm}{Theorem}[section]         
\newtheorem{Lemma}[Thm]{Lemma}
\newtheorem{Cor}[Thm]{Corollary}
 
\newtheorem{Prop}[Thm]{Proposition}

\theoremstyle{definition}
\newtheorem{Def}[Thm]{Definition} 
\newtheorem{Rem}[Thm]{Remark}

\newtheorem{Ex}[Thm]{Example}

\numberwithin{equation}{section}

\begin{document}

\title{Local and global robustness at steady state}
\author{B. Pascual-Escudero, E. Feliu}
\date{\today }

\footnotetext[1]{\textsc{Department of Mathematical Sciences, University of Copenhagen} Universitetsparken 5, 2100 Copenhagen, Denmark. \par
\textit{E-mail address}, B. Pascual-Escudero: \texttt{beatriz@math.ku.dk}  \par
  \textit{E-mail address}, E. Feliu: \texttt{efeliu@math.ku.dk} \par
  }
  
\maketitle
\begin{abstract}
We study the robustness of the  steady states of a class of systems of autonomous ordinary differential equations (ODEs), having as a central example those arising from (bio)chemical reaction networks. More precisely, we study under what conditions  the steady states of the system are contained in a parallel translate of a coordinate hyperplane. 

To this end, we focus mainly on ODEs consisting of generalized polynomials, and make use of algebraic and geometric tools to relate the local and global structure of the set of steady states.
Specifically, we consider the local property termed \emph{zero sensitivity} at a coordinate $x_i$, which means that the tangent space is contained in a hyperplane of the form $x_i=c$, and provide a criterion to identify it. We consider the global property termed \emph{absolute concentration robustness (ACR)}, meaning that all steady states are contained in a hyperplane of the form $x_i=c$. 

We clarify and formalise the relation between the two approaches. In particular, we show that ACR implies zero sensitivity, and identify when the two properties do not agree, via an intermediate property we term \emph{local ACR}. For families of systems arising from modelling biochemical reaction networks, we obtain the first practical and automated criterion to decide upon (local) ACR.
 \end{abstract}

\section{Introduction}
A crucial property of many biological systems is their capacity to maintain specific features against environmental or structural perturbations. For instance, the concentration of certain species in some biochemical systems stabilise to a fixed value upon changes in the concentrations of the other species (a property commonly known as adaptation \cite{S-F,barkai:robustness,yi:robustness,alon:robustness}).

 Formal investigation and characterization of this type of \emph{robustness} resides in the study of associated mathematical models at steady state, and on how the set of steady states depends on the input conditions. In the deterministic setting, models based on differential equations take the form $\tfrac{dx}{dt}= f_k(x)$, where $x\in \R^n_{>0}$ is the vector of concentrations, $k\in \R^r_{>0}$ a parameter vector and $f_k$ is a vector of continuously differentiable functions.
It is often the case that the dynamics of these systems are confined to invariant linear subspaces of the form $Wx=T$ for a matrix $W\in \R^{d\times n}$ and $T\in \R^d$, depending on the initial condition. 
 Then robustness can refer to variations with respect to  (some of)  the parameters $k$ of the system, or the parameters $T$ defining the linear subspace. In this work we focus on the latter.

\smallskip
Robustness is addressed following two different approaches: \emph{local} or \emph{global}.
The global approach investigates the property of having identical concentration of a certain species at any \emph{positive} steady state for all values of the parameter under consideration. Global robustness with respect to $T$ for a certain type of polynomial systems has been introduced in \cite{S-F} and termed \textit{Absolute Concentration Robustness (ACR)}. 
In that work, a simple sufficient criterion to decide upon ACR  for a (small) class of systems is presented.
   Further works \cite{Karp,P-M} explore ACR in systems where $f_k$ is polynomial, and the difficulty to establish a uniform criterion to determine whether a system presents ACR, emerges. The main problem resides in the study 
  of the zero set of $f_k$ in $\R^n_{>0}$, that is, of the intersection of the algebraic variety defined by $f_k$ with the positive orthant, for unknown $k\in \R^r_{>0}$.
   In this work we relax the definition of ACR by allowing the concentration of the species 
   at any positive steady state to take one of a finite number of values. We term this property \emph{local ACR} and it turns out that it is easier to check than ACR for generalized polynomial systems.

 The local approach focuses on the effect that a slight perturbation of the parameter has on the value of the concentration of the species at steady state. When the perturbation is infinitesimally small, then this corresponds to taking the derivative of the concentration of the species at steady state with respect to the perturbation (to be formalized in Section~\ref{sec:ZeroSensitivity}). 
 In the context of biological systems, these derivatives are termed \emph{sensitivities} and are commonly employed to quantify the degree of robustness, mainly in metabolic analysis \cite{Fiedler-Mochizuki,SMJF,OM16,BF_3,varma_morbidelli_wu_1999,steuer:robustness}. Full local robustness corresponds to the situation where the derivatives  with respect to all perturbations  are zero, and we say that the system has \emph{zero sensitivity} in $x_i$ at a given steady state. 
 
 \smallskip
 Motivated by these questions from biological systems, we study robustness for generic systems of equations 
 $g(x)=0$, $x\in \R^n_{>0}$, with $g=(g_1,\dots,g_s)$ a vector of generalized polynomials, intersected with linear subspaces $Wx=T$ as above. The system has  local ACR with respect to $x_i$ if and only if $x_i$ attains a finite number of values for any solution to $g(x)=0$.
We provide a simple criterion to decide upon local ACR, which under certain conditions, gives also zero sensitivity with respect to $T$. Clearly, ACR necessitates local ACR, hence we provide a necessary condition for ACR. Furthermore, the relation to zero sensitivity shows how the local property of zero sensitivity is related to ACR. 

Formally, let $\Vp$ be the set of positive real solutions to $g(x)=0$.    The system can only have ACR with respect to $x_i$ if for every $c,\tilde{c}\in \Vp$, the difference $c_i-\tilde{c}_i$ is zero. If $c$ is a regular point, then $\Vp$ is locally a  differential real manifold, and hence the system can only have ACR if  for all regular points $c$, the $i$-th entry of any vector tangent to $\Vp $ at   $c$ is zero. 
The latter is precisely the definition of zero sensitivity in $x_i$.

Conversely, if the system has zero sensitivity in $x_i$, the tangent space  at every point of $\Vp$ is contained in a hyperplane normal to the $i$-th canonical vector of $\mathbb{R}^n$. Under certain conditions on the regularity of $\Vp$, this happens for all regular points if and only if the value of $x_i$ is constant in each irreducible component of $\Vp$, or equivalently, as we will see, if and only if the system displays local ACR with respect to $x_i$.
This is the intuition behind the following main result, which is a version of Theorem~\ref{thm:6_5}. A non-degenerate solution $x^*$ to a system $f(x)=0$, is a solution where the Jacobian matrix $\frac{\partial f(x^*)}{\partial x}$ has maximal rank. 

\medskip
\textbf{Theorem A.} \emph{
Let $g=(g_1,\dots,g_s)$ be such that  $g_1,\dots,g_s$ are linearly independent generalized  polynomials in $\R^n$, and let $W\in \R^{(n-s)\times n}$ be of maximal rank. Assume that $\Vp\neq \emptyset$ and that all irreducible components of $\Vp$ contain  a non-degenerate solution to the system $g(x)=0$, $Wx-T=0$ for some $T$. The following statements are equivalent:
\begin{enumerate}
	\item The system $g(x)=0$ has local ACR with respect to $x_i$.
	\item The system $g(x)=0$ has zero sensitivity in $x_i$ at all 
	non-degenerate solutions to the system $g(x)=0$, $Wx-T=0$ (for some $T$).
\item The rank of the $s \times (n-1)$ matrix $(\frac{\partial g(x^*)}{\partial x})^i$, 
	 resulting from eliminating the $i$-th column of $\frac{\partial g(x^*)}{\partial x}$,  is smaller than $s$ for all $x^*\in \Vp$.
\end{enumerate}
}

The equivalence between (1) and (3) in Theorem A holds also under the milder  assumption that all irreducible components of $\Vp$  have a non-degenerate solution to the system $g(x)=0$.

\medskip
In the application to models of the concentration of the species in a reaction network,  
we are typically  interested in a parametric family of functions $g_k$ for $k\in \R^r_{>0}$.
In this case, we obtain a simpler criterion for the 
  existence of local ACR for the whole family of systems obtained when varying  $k$.  Denoting  by $B^t$ the transpose of a matrix $B$, a simplified version of this result is as follows (Theorem~\ref{thm:main_b}).

\smallskip
\textbf{Theorem B.} \emph{
Let $g_k(x)= N \diag(k) x^B$ with $N\in \R^{s\times n}$ of rank $s$, $B\in \Q^{n\times r}$, and $k\in \Rp ^r$.
 Assume 
that the matrix $N\mathrm{diag}(v)B^t$ has rank $s$ for all $v\in \mathrm{ker}(N)\cap \R^r_{>0}$.}

\emph{
Then the system $g_k(x)=0$  has local ACR with respect to $x_i$  for all  $k\in \Rp ^r$  if and only if 
  all ($s\times s$)-minors of the matrix $N\mathrm{diag}(v)B^t$ not involving column $i$ are  zero for all $v\in \mathrm{ker}(N)$.}

 \smallskip

Theorem B gives an easy-to-check necessary condition  for ACR, valid whenever $\Vp^k$ consists of  non-degenerate points to $g_k(x)=0$.  Hence, we can automatically search for ACR  in numerous families of systems 
 by simply checking a linear algebra condition. A systematic exploration of the occurrence of ACR in biochemical systems will be  presented in a  subsequent upcoming work.

 \smallskip
The paper is organized as follows: In Section \ref{sec:framework} we establish the notation and the type of systems we consider, by focusing on the motivation from dynamical systems and reaction networks.
In Section \ref{sec:ZeroSensitivity} we introduce the concept of zero sensitivity and a criterion to decide upon that (Theorem \ref{thm:k_fix_rankVSsens0}). 
We proceed in Section~\ref{sec:ACR}  to define ACR and local ACR and explore ways to determine local ACR. 
As this part requires a heavier use of real algebraic geometry, we keep Section~\ref{sec:ACR} expository by presenting the results with minimal technicalities, and the details are postponed to Section \ref{sec:backgroundAG}, for readers acquainted with the algebraic geometry language.
In Section~\ref{subsec:lACR_via_sens}, the connection between zero sensitivity and local ACR is given (Theorem~\ref{thm:6_5}), and the specific situation of systems arising from reaction networks is studied (Theorem~\ref{thm:main_b}). Examples are provided throughout to illustrate the concepts and results.

\section{Framework}\label{sec:framework}

Throughout we denote by $\mathcal{C}^1(\Omega, \Omega')$, with $\Omega\subseteq \R^n$ open and $\Omega'\subseteq \R^s$, the set of continuously differentiable functions $\Omega \rightarrow \Omega'$, and by $\frac{\partial {f}}{\partial x}\in \R^{s\times n}$   the Jacobian matrix of $f\in \mathcal{C}^1(\Omega, \Omega')$. 
 Given a matrix $A\in \R^{s\times n}$ and indices $i\in \{1,\dots,n\}$, $j\in \{1,\dots,s\}$, we denote by $A_j^i$ the matrix obtained by removing the $j$-th row and $i$-th column of $A$. 
We denote simply by $A^i$ the matrix obtained by removing the $i$-th column of $A$. 
We let $0_{s\times d}$ denote the $s\times d$ matrix with zero entries and $\mathrm{Id}_{d\times d}$   the identity matrix of size $d$.

We use $\langle v_1,\dots,v_s \rangle$ to denote the vector subspace generated by vectors $v_1,\dots,v_s\in \R^n$, or the ideal generated by polynomials $v_1,\dots,v_s\in \R[x_1,\dots,x_n]$, depending on the context.

\subsection{Motivation: steady states and reaction networks.}\label{sec:motivation}
Consider systems of autonomous ordinary differential equations (ODEs) in $\R^n$ of the form
\begin{equation}\label{eq:ODE2}
\tfrac{dx}{dt} = f(x),\qquad x\in \Omega,
\end{equation}
for $ \Omega\subseteq \R^n_{>0}$, and such that $f\in \mathcal{C}^1(\Omega,\R^n)$. Here $x=x(t)$ and reference to $t$ is omitted. 
Given such a system, we consider the vector subspace generated by the image of $f$: 
\begin{equation}\label{eq:S} S_f := \big\langle (f_1(x),\dots,f_n(x)) \, \mid\,  x\in \Omega \big\rangle \subseteq \R^n. \end{equation}
Let $s=\dim(S_f)$.   By construction, the trajectories of \eqref{eq:ODE2}  are confined to 
 the affine linear subspaces $x^0+S_f$, where $x^0$ is the initial condition (the cosets of $S_f$). 
For any $\omega\in S_f^\perp$, $\omega\cdot x$ is a linear first integral for \eqref{eq:ODE2}.  Linear first integrals appear commonly in systems arising from modelling (bio)chemical reaction networks, due to conserved moieties (Example \ref{ex:canonical} below). Their existence is not necessary to study local ACR, but play a role in our definition of zero sensitivity.  Equations 
of the cosets of $S_f$ are given by 
\begin{equation}\label{eq:cons} 
W x = T,\qquad T\in \R^{n-s},
\end{equation}
for  $W\in \R^{(n-s)\times n}$ any matrix of full rank whose rows of $W$ form a basis of $S_f^\perp$. Here $T=Wx^0$, if $x^0$ is the initial condition.

\medskip
Given an ODE system as in \eqref{eq:ODE2}, the \emph{steady states} or \emph{equilibrium points} are the solutions to the equation 
\begin{equation}\label{eq:ss}
f(x)=0,\qquad x\in \Omega.  
\end{equation}
Considering \eqref{eq:S}, $n-s$ of the equations in \eqref{eq:ss} are redundant. 
Hence \eqref{eq:ss} is equivalent to a system $g(x)=0$ with $g=(g_{1},\ldots ,g_{s})$, and such that $S_g=\mathbb{R}^s$.

Note that if $f$ is polynomial, or a generalized polynomial where exponents are allowed to be real numbers, then $\dim S_f$ agrees with the rank of the coefficient matrix of $f$.

\begin{Ex}[Chemical Reaction Networks] \label{ex:canonical}
The main scenario in which the setting above applies is that of chemical reaction network theory \cite{feinberg-book}. A \textit{(chemical) reaction network} on a set of \textit{species} $\{ X_1,\ldots ,X_n\} $ is a digraph, where nodes are linear combinations of the species with non-negative integer coefficients. 
We use $r$ for the number of edges of the graph.
Each edge represents a \textit{reaction} and is of the form
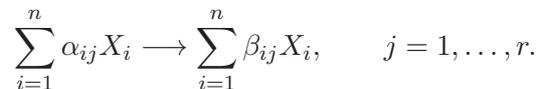
\[  \sum _{i=1}^{n} \alpha _{ij}X_i \longrightarrow \sum _{i=1}^{n} \beta _{ij}X_i,\qquad j=1,\dots,r.\]
The \textit{stoichiometric matrix} $\Gamma\in \R^{n\times r}$  is defined such that each column encodes the net production of the different species in a reaction: $\Gamma_{ij} = \beta _{ij}-\alpha _{ij}$.

We let $x(t)=(x_1(t),\ldots ,x_n(t))$  be the vector of concentrations of the species $X_1,\dots,X_n$  at time $t$. 
Given a \emph{kinetics} $\nu  \in \mathcal{C}^1(\R^n_{>0},\R^r_{\geq 0})$, the evolution of the concentration of the species over time is described by means of an ODE of the form:
\begin{equation*}\label{eq:ODE}
\tfrac{dx}{dt}=\Gamma \nu (x), \qquad x\in \R^n_{>0}.
\end{equation*}
Hence, in this setting, we have $\Omega=\R^n_{>0}$ (typically, the image of $\nu$ will be in $\R^n_{>0}$.)

So-called \emph{power-law} kinetics arise when $\nu (x)= \diag(k) x^B$ for a matrix $B\in \R^{n\times r}$ and a vector of \textit{reaction rate constants} $k\in \R^r_{>0}$. Here $x^B$ is the vector of monomials with $j$-th entry corresponding to the monomial arising from the $j$-th column of $B$ as
\[ \big(x^B\big)_j = \prod_{i=1}^n x_i^{b_{ij}}. \]
For every $k\in \R^r_{>0}$, we have an ODE system $\tfrac{dx}{dt}=f_k(x)$ where
\begin{equation}\label{eq:powerlaw}
f_k(x)= \Gamma \diag(k) \, x^B, \qquad x\in  \R^n_{>0}.
\end{equation}

Among these kinetics, the main choice arises under the \emph{mass-action} assumption, where $B$ is defined by the coefficients of the complexes at the left of each reaction (the reactants): $b_{ij} = \alpha_{ij}$.
 
In the chemical literature, the equations \eqref{eq:cons} define the so-called \emph{stoichiometric compatibility classes}. 
Under realistic assumptions, namely that each connected component of the network has a unique terminal strongly connected component \cite{feinberg-invariant}, $S_{f_k}={\rm im}(\Gamma)$ and hence is independent of  $k$.
In this case, the rows of the matrix $W$ in \eqref{eq:cons} form a basis of the left kernel of $\Gamma$.

As mentioned in the introduction,  we aim at understanding when ACR arises, that is, in determining whether the set of steady states of reaction networks is contained in a hyperplane of the form $x_i-C=0$. Abstracting from this motivating scenario, the question of interest is the general problem of whether the solution set to a system $g(x)=0$ is contained in such a hyperplane, and therefore is studied with full generality here, without restricting to  the specific scenario of $g$ being derived from the right hand-side of an ODE system. 
\end{Ex}

\subsection{Setting. }\label{sec:setting}
Our starting point for the rest of the paper will be a   function 
\[ g=(g_1,\dots,g_s), \qquad  g\in \mathcal{C}^1(\Omega, \R^s),\]
where $\Omega \subseteq \mathbb{R}^n$, such that $S_g=\mathbb{R}^s$ (remove linearly dependent entries if this is not the case, see Subsection~\ref{sec:motivation}). 
Given a vector subspace $S\subseteq \R^n$ of dimension $s$,   let $d=n-s$ and consider a full rank matrix $W\in \R^{d\times n}$ whose rows form a basis of $S^\perp$ such that  $S=\ker(W)$. For $T\in \R^d$, we define the function
\begin{equation}\label{eq:F_T}
F_{T}(x)=\left( \begin{array}{c} g(x) \\ Wx-T\end{array}\right), \qquad x\in \Omega.
\end{equation}
 In the context of Subsection~\ref{sec:motivation}, a solution to $F_T(x)=0$ in $\Omega$ is a steady state of \eqref{eq:ODE2} in the coset of $S$ with equation $Wx=T$.  
 
 Consider the set of solutions to the system $g(x)=0$ 
  \begin{equation}\label{eq:Vp}
 \Vp := \{ x\in \Omega \mid  g(x)=0\}. 
 \end{equation}

In what follows, solutions to a system where the rank drops,  which we we refer to as degenerate,  play a central role. 

\begin{Def}[Degenerate solution]\label{def:deg} 
Let $g$, $\Vp$, $S$, $W$, $T$ and $F_T$ as above. 
\begin{itemize}
\item A solution $x^*\in \Vp$ of the system $g(x)=0$ is  \textit{degenerate}   if $\mathrm{rank}\big(\frac{\partial g(x^*)}{\partial x}\big)<s$. 
\item A solution $x^*\in \Vp$ of the system $g(x)=0$ is \textit{degenerate with respect to $S$}, if 
it is degenerate for  the system $F_T(x)=0$ with $T=Wx^*$. 
\end{itemize}
  We let $\Deg (g)$, resp. $\Deg_S(g)$ denote the subsets of $\Vp$ consisting of degenerate solutions, resp. degenerate solutions with respect to $S$. Similarly, we write 
\[ \NDeg (g) := \Vp \setminus \Deg (g) \qquad \NDeg_S (g) := \Vp \setminus \Deg_S (g)\]
for the sets of non-degenerate solutions, resp. non-degenerate with respect to $S$.
\end{Def}

The definition of degeneracy with respect to $S$ is independent of the 
choice of matrix $W$.  Specifically, it is equivalent to the statement
$ \mathrm{ker}\big(\tfrac{\partial g(x^*)}{\partial x}\big)\cap S\neq \{0\},$
as $S=\ker(W)$.  Furthermore, the Jacobian of $F_T$ does not depend on $T$.
It follows readily from the definition that
\begin{equation}\label{eq:inclusion_deg}
\Deg(g)\subseteq \Deg _S(g).
\end{equation}
If $x^*\in  \Deg _S(g)\setminus \Deg(g)$, then
  $\Vp$ intersects tangentially $x^*+S$ at $x^*$. Specifically, $x^*$ is non-degenerate with respect to $S$ if and only if $\tilde{S}(x^*)\oplus S^{\bot}=\mathbb{R}^n$, where $\tilde{S}(x^*)\subseteq \mathbb{R}^n$ is the vector subspace generated by the rows of $\frac{\partial {g}(x^*)}{\partial x}$.   
  
  Definition~\ref{def:deg} extends  to all points $x^*$ where $g$ is defined, also outside the   domain $\Omega$. In particular, when $g$ is polynomial, we will use, without further considerations, complex degenerate points, found by considering the extension of $S$ as a vector subspace of $\C^n$.

\medskip
\begin{Rem}
Two systems can have the same solution set but different degenerate points. Hence, degeneracy is not only a property of the set of solutions, but can depend on the equations.  For instance, 
the systems $x=0$ and $x^2=0$ have the same solution, namely $x^*=0$. However, $0$ is degenerate for the second system but not for the first.  
\end{Rem}

\section{Zero sensitivity}\label{sec:ZeroSensitivity}

 We adopt the notation and setting of Subsection~\ref{sec:setting}. 
 The first measure of robustness for the points of $\Vp$  in \eqref{eq:Vp} considers how changes in $T$ affect the points. To be precise, each point $x^*$ of $\Vp$ belongs to exactly one of the linear varieties $Wx=T$ for $T= W x^*$. If $x^*$ is non-degenerate with respect to $S$, then $Wx=T$ is transversal to $\Vp$ at $x^*$, and hence small perturbations to $T$ lead to a new point of $\Vp$. This gives rise to the notion of \emph{sensitivity}, which has been studied in different forms and scenarios in several works in the context of reaction networks \cite{SMJF,OM16,BF_3,varma_morbidelli_wu_1999,steuer:robustness}. When the system arises from a reaction network as in Example~\ref{ex:canonical},  then $T$ may be perturbed as a result of the perturbation of the initial concentration of some species, which might be controllable in experiments (see Remark \ref{rem:onlyTnec}).  This fact motivates the study of sensitivities with respect to $T$. 
  
 As a measure of robustness, we consider the case where one of the entries of $x^*$ remains constant upon infinitesimally small changes in $T$, and we refer to this property as \emph{zero sensitivity}. 
  
\subsection{Sensitivities}\label{subsec:sensitivities}
We follow the formalism from \cite{Fel} on sensitivities with the setting of Subsection~\ref{sec:setting}. 
Consider $g\in \mathcal{C}^1(\Omega,\R^s)$ with $\Omega\subseteq \R^n$, a vector subspace $S\subseteq \R^n$ of dimension $s$, and a matrix $W\in \R^{d\times n}$ of maximal rank $d=n-s$ such that $S=\ker(W)$.  Although we do not require $\dim S_g= \R^s$  in the forthcoming constructions, when this does not hold, 
$\NDeg_S(g)=\emptyset$, and the assumptions for the results in this section do not hold.

Given  $x^*\in \NDeg_S(g)$,
let $T^*=W x^*$, and consider a perturbation of $T^*$  given by 
a continuously differentiable function 
$\gamma  \colon (-\epsilon, \epsilon)      \rightarrow  \R^d$  for
 $\epsilon>0$ such that $\gamma(0)=T^*$. 
By the non-degeneracy of $x^*$,  $\frac{\partial F_{T^*}(x^*)}{\partial x}$ has rank $n$. Hence, the Implicit Function Theorem applied to $F_{\gamma(u)}(x)$ in the variables $(u,x)$, gives the existence of $0<\delta<\epsilon$ and a continuously differentiable curve  in a neighborhood of $0$ 
 \begin{align}\label{eq:curve_c}
c\colon (-\delta , \delta ) & \longrightarrow  \Omega\end{align}
satisfying $c(0)=x^*$ and $F_{\gamma (u)}(c(u))=0$ for every $u\in (-\delta , \delta )$. 
The vector $c'(0)$ is called the \textit{sensitivity} of  $x_1,\ldots ,x_n$ with respect to the perturbation $\gamma$ at $x^*$. 
To keep track of the perturbation under consideration and the point $x^*$, we denote this vector by
\[ \Sen _{\gamma }(x^*)=(\Sen _{\gamma ,1}(x^*),\ldots ,\Sen _{\gamma ,n}(x^*)). \]
By differentiating the equation $F_{\gamma (u)}(c(u))=0$ with respect to $u$ and evaluating at $u=0$, we obtain that $\Sen _{\gamma }(x^*)$ is the solution to the system
\begin{equation}\label{eq:sys_sens}
\frac{\partial F_{T^*}(x^*)}{\partial x}\Sen _{\gamma }(x^*)+\left( \begin{array}{cc} 0_{(n-d)\times d} \\ -\mathrm{Id}_{d\times d}\end{array}\right) \gamma'(0)=0,
\end{equation}
where $\gamma '=\frac{\partial \gamma }{\partial u}$. 
See \cite{Fel} for details.   Equation~\eqref{eq:sys_sens} illustrates that  $\Sen _{\gamma }(x^*)$ depends on the choice of $W$,  as the last $d$ rows of $\frac{\partial F_{T^*}(x^*)}{\partial x}$  agree with $W$. See  the proof of Lemma~\ref{prop:eq_ST0_any} for details. 
 
\begin{Rem}\label{rem:onlyTnecAdd}
The sensitivity vector  $\Sen _{\gamma }(x^*)$ at a point $x^*\in \NDeg_S(g)$ with respect to any perturbation
 of $T^*$ is the derivative of a curve in $\VS$ at $x^*$. Thus, it is contained in the tangent space $\ker\big( \frac{\partial g(x^*)}{\partial x} \big) $ to $\Vp $ at $x^*$. This tangent space has dimension $d$ and  is generated by the sensitivities with respect to the \emph{canonical perturbations}
\begin{align*}\label{eq:gamma_tj}
\gamma _{j} \colon (-\epsilon, \epsilon) & \longrightarrow  \R^d \\
u & \longmapsto T^*+ u\, e_j   \nonumber
\end{align*}
for $j=1,\ldots ,d$, where $e_j$ is the $j$-th element of the standard basis of $\mathbb{R}^d$.  In this case $\gamma'_j(0)=e_j$
and $\Sen _{\gamma _j}(x^*)$ is the solution to the system
\begin{equation}\label{eq:sys_perturb_jT}
\tfrac{\partial F_{T^*}(x^*)}{\partial x}\, \Sen _{\gamma _j}(x^*)=\left( \begin{array}{c} 0_{(n-d)\times 1} \\ e_j\end{array}\right).\end{equation} 
In particular, for an arbitrary perturbation $\gamma$,  it holds that $\Sen _{\gamma }(x^*)=\sum_{i=1}^d\gamma'(0)_i\, \Sen _{\gamma_i }(x^*)$.
\end{Rem}

\begin{Rem}
For perturbations of the form $u \mapsto T^*+u\, v $ with $v\in \R^d$ (including canonical perturbations),  $\gamma'(0)=v$ does not depend on $T^*$. As $\frac{\partial F_{T^*}(x^*)}{\partial x}$ does not either depend on $T^*$, system \eqref{eq:sys_sens} has no explicit dependence on $T^*$ and hence sensitivities do not depend directly on the value of $T^*$ only indirectly as $Wx^*=T^*$. Observe that this does not hold for all perturbations, for example for $\gamma(u)=(u+1)T^*$, where $\gamma'(0)=T^*$.

\end{Rem}

\subsection{Zero sensitivity with respect to $S$}\label{subsec:zero_sensitivity}

We introduce now the concept of zero sensitivity, meaning that one of the components of $\Sen_\gamma(x^*)$ vanishes for all perturbations $\gamma$.

\begin{Def}\label{def:zero_sens}
Let  $g\in \mathcal{C}^1(\Omega,\R^s)$ with $\Omega\subseteq \R^n$, $S\subseteq \R^n$  a vector subspace of dimension $s$, and $W\in \R^{d\times n}$ a matrix of maximal rank $d=n-s$ such that $S=\ker(W)$. The system $g(x)=0$ 
has \textit{zero sensitivity in the variable $x_i$} (with respect to $W$) at a point $x^*\in \NDeg_S(g)$ if $\Sen _{\gamma,i}(x^*)=0$ for any perturbation $\gamma$ of  $T^*=Wx^*$.
\end{Def}

\begin{Rem}\label{rem:n=s_zero}
If $s=n$, then $d=0$ and $F_T=g$. All sensitivities are then vacuously zero with respect to any perturbation, and thus the system has zero sensitivity in all variables at all non-degenerate points.  In fact, if $d=0$, then the non-degenerate steady states are simply isolated points, and the curve $c$ in \eqref{eq:curve_c} is constant at  $x^*$.
\end{Rem}

In virtue of Remark \ref{rem:onlyTnecAdd}, it is enough to consider canonical perturbations: The system $g(x)=0$ has zero sensitivity at a point $x^*$ with respect to $W$ if and only if $\Sen _{\gamma _j,i}(x^*)=0$ for all $j=1,\dots,d$. 
Using \eqref{eq:sys_perturb_jT} and Cramer's rule, it is straightforward to compute $\Sen _{\gamma _j }(x)$ for all $j=1,\ldots ,d$.

\begin{Lemma}\label{lemma:sens} 
Let $g$ and $W$ be as in Definition~\ref{def:zero_sens}, and $F_T$ be as in \eqref{eq:F_T}. Consider $x^*\in \NDeg_S(g) $ and $T^*=Wx^*$.  
The sensitivity of $x_i$ with respect to the canonical perturbation $\gamma _{j}$  at $x^*$ is:
\begin{equation*}\label{eq:S_T}
\Sen _{\gamma _j,i}(x^*)=\frac{(-1)^{i+n-d+j}\, \mathrm{det}\hspace{-0.08cm}\left( \left( \frac{\partial F_{T^*}(x^*)}{\partial x}\right) ^i_{n-d+j}\right) }{\mathrm{det}\hspace{-0.08cm}\left( \frac{\partial F_{T^*}(x^*)}{\partial x}\right) },\qquad j=1,\ldots ,d.
\end{equation*}
\end{Lemma}

\medskip

Definition~\ref{def:zero_sens} assumes that  the matrix $W$  is fixed. In fact, as noticed in \cite{Fel}, the choice of $W$ might have consequences in the value of $\Sen _{\gamma _j}(x)$, as two matrices that do not differ in the $j$-th row might give rise to different sensitivity vectors for the $j$-th canonical perturbation. However, as we show in the next lemma, zero sensitivity does not depend on the choice of $W$, and depends only on $S$. Hence, it makes sense to refer to \emph{zero sensitivity with respect to $S$. }

\begin{Lemma}\label{prop:eq_ST0_any}
Let   $g\in \mathcal{C}^1(\Omega,\R^s)$ with $\Omega\subseteq \R^n$, $S\subseteq \R^n$, 
$x^*\in \NDeg_S(g)$ and $W,W'\in \R^{d\times n}$ of maximal rank $d=n-s$ such that $S=\ker(W)=\ker(W')$.
The system $g(x)=0$ 
has \textit{zero sensitivity in the variable $x_i$} with respect to $W$ if and only if it has \textit{zero sensitivity in the variable $x_i$} with respect to $W'$.

\end{Lemma}

\begin{proof}
Let $A\in \R^{d\times d}$ be an invertible matrix such that $W=AW'$. For $T=Wx^*$ and $T'=W'x^*$, we have
{\small \begin{equation}\label{eq:A}
\frac{\partial F_{T}(x^*)}{\partial x}=\left( \begin{array}{cc} Id_{s\times s} & 0_{s\times d} \\ 0_{d\times s} & A\end{array}\right) \frac{\partial F_{T'}(x^*)}{\partial x}. 
\end{equation}}
We let $\gamma_1,\dots,\gamma_d$ be the canonical perturbations of $T$ for the matrix $W$, and 
$\gamma'_1,\dots,\gamma'_d$ be the canonical perturbations of $T'$ for the matrix $W'$. Then, for $j\in \{ 1,\ldots ,d\} $, the sensitivity vectors $\Sen _{{\gamma_j}}(x^*)$, $\Sen _{{\gamma'_j}}(x^*)$ are respectively given as
{\small \[ 
\Sen _{\gamma_j}(x^*)=\left(\frac{\partial F_{T}(x^*)}{\partial x}\right)^{-1}\left( \begin{array}{c} 0_{s\times 1} \\ e_j\end{array}\right) \quad \textrm{and}\quad 
\Sen _{\gamma'_j}(x^*)=\left(\frac{\partial F_{T'}(x^*)}{\partial x}\right)^{-1}\left( \begin{array}{c} 0_{s\times 1} \\ e_j\end{array}\right). 
\]}
Using \eqref{eq:A}, we also have
{\small \[ \Sen _{\gamma'_j}(x^*) = \left(\frac{\partial F_{T}(x^*)}{\partial x}\right)^{-1}\left( \begin{array}{c} 0_{s\times 1} \\ Ae_j\end{array}\right). 
\]}%
As $A$ is invertible, the sets {\small $\left\{ \left( \begin{array}{c} 0_{s\times 1} \\ e_j\end{array}\right)\right\}_{j=1,\dots,d} $} and {\small $\left\{ \left( \begin{array}{c} 0_{s\times 1} \\ A e_j\end{array}\right)\right\}_{j=1,\dots,d} $} are bases of the same vector subspace $V$. 
Let $v$ be the $i$-th row of $ \left(\frac{\partial F_{T}(x^*)}{\partial x}\right)^{-1}$. Then 
$\Sen _{\gamma_j,i}(x^*)=0$ or $\Sen _{\gamma'_j,i}(x^*)=0$ for all $j=1,\ldots ,d$ are both equivalent to $v\in V^\perp$. 
\end{proof}

\begin{Rem}\label{rem:onlyTnec} 
As in \cite{Fel}, we might consider perturbations of $T$ obtained by perturbing the  point $x^*$. In the motivating setting of reaction networks, this might correspond to the addition or removal of a small amount of one of the species. Such a perturbation $\lambda\colon (-\epsilon,\epsilon) \rightarrow \R^n_{>0}$ with $\lambda(0)=x^*$  gives rise to a perturbation $\gamma\colon  (-\epsilon,\epsilon) \rightarrow \R^d$ of $T^*=Wx^*$ defined by $\gamma(s)= W(\lambda(s))$. As $W$ has full rank, any perturbation of the form $T^* \mapsto T^* + u\, v$  for $v\in \R^d$
arises as a perturbation of the form $x^* \mapsto x^* + u \, \overline{v}$ for $\overline{v}\in \R^n$ by this construction.

Studying zero sensitivity perturbing $T$ or $x^*$  is equivalent. Specifically, consider the perturbations of $x^*$ given by  $\lambda_j(s)= x^*+ s e_j$, where now $e_j$ is the $j$-th canonical vector of $\R^n$.
Then, $g(x)=0$ has zero sensitivity in the variable $x_i$ at $x^*\in \NDeg_S(g)$ if and only if $\Sen _{\gamma ,i}(x^*)=0$ for all perturbations $\gamma$ arising as above from $\lambda_j$, for $j=1,\dots,n$.

A  perturbation of a variable $x_i$ that  is not a variable of  \eqref{eq:cons} or, equivalently, such that the $i$-th column of $W$ is identically $0$ (this fact is independent of the choice of $W$), induces the constant perturbation of $T$, hence $\gamma'(0)=0$. Then, by \eqref{eq:sys_sens}, $\Sen _{\gamma }(x^*)=0$ for any $x^*\in \NDeg_S(g)$. 
\end{Rem}

The computation of the sensitivities by solving equation \eqref{eq:sys_sens} in Lemma~\ref{rem:n=s_zero}  is well established and has been presented in different contexts.  For example, an analogous computation for sensitivities is found within Ecology in \cite{Nakajima}, where systems with no conservation laws are considered and perturbation is with respect to other parameters of the system.

\subsection{Determining zero sensitivity.}
The main result of this section is a criterion that allows us to determine when a system has zero sensitivity in a variable at a point that is non-degenerate with respect to $S$, by simply inspecting the matrix $\frac{\partial g(x^*)}{\partial x}$. This criterion will be  key to the results in Section \ref{subsec:lACR_via_sens}. Recall that $\big( \tfrac{\partial g(x^*)}{\partial x}\big)^i$ is obtained   by removing the $i$-th column of $ \tfrac{\partial g(x^*)}{\partial x}$.

\begin{Thm}\label{thm:k_fix_rankVSsens0} 
Let  $g\in \mathcal{C}^1(\Omega,\R^s)$ with $\Omega\subseteq \R^n$, $S\subseteq \R^n$  a vector subspace of dimension $s$, 
and $x^*\in  \NDeg_S(g)$.
 The system $g(x)=0$  has zero sensitivity in the variable $x_i$ with respect to $S$ at $x^*$ if and only if 
  \begin{equation*}\label{eq:rank_dim3}
 \rank\Big( \big( \tfrac{\partial g(x^*)}{\partial x}\big)^i \Big)<  s.
  \end{equation*}
 
 \smallskip
In particular, zero sensitivity does not depend on   $S$, as long as $x^*\in \NDeg_S(g)$, that is, $S$ is transversal to $\ker\big(\tfrac{\partial g(x^*)}{\partial x}\big)$ and $\tfrac{\partial g(x^*)}{\partial x}$ has rank $s$.
\end{Thm}

\begin{proof}
Without loss of generality, we assume that $i=n$. Let $W\in \R^{d\times n}$ such that $S=\ker(W)$ and $F_T$ the associated map as in \eqref{eq:F_T}. 

 For $\ell=1,\ldots ,n$, let $u_\ell$ stand for the $\ell$-th column of the matrix $\frac{\partial g(x^*)}{\partial x}$, and let $w_\ell$ stand for the $\ell$-th column of the matrix $W$. By Lemma~\ref{lemma:sens}, for $j=1,\ldots ,d$, $\Sen _{\gamma _j,n}(x^*)=0$ if and only if 
\begin{equation}\label{eq:vanish}
0= (-1)^{2n-d+j}\mathrm{det} \left( \left( \tfrac{\partial F_{T}(x^*)}{\partial x}\right) _{n-d+j}^n\right) =\mathrm{det}\left( \begin{array}{cccccc} u_1 & \cdots & u_{n-1} & 0 \\ w_1 & \cdots & w_{n-1} & e_j \end{array}\right) \mbox{,}
\end{equation}
 
 As $x^*\in \NDeg_S(g)$, the rank of $\frac{\partial F_{T}(x^*)}{\partial x}$ is maximal. Hence  
 the column vectors in the set  {\small $U=\left\{ \left( \begin{array}{c} u_\ell \\ w_\ell \end{array}\right) \right\} _{\ell=1}^{n-1}$ }are linearly independent. 
  Then \eqref{eq:vanish} holds if and only if all column vectors {\small $\left( \begin{array}{c} 0 \\ e_j \end{array}\right)$} for $j=1,\ldots ,d$ are linear combinations of $U$. The statement of the theorem is  thus equivalent to the following statement:
\small
\begin{equation}\label{eq:thm_eq}\mathrm{rank}  \left( \left( \begin{array}{ccc} u_1 &  \cdots & u_{n-1} \end{array}\right)  \right) < n-d \Leftrightarrow \mathrm{rank} \left( \left( \begin{array}{cccccc} u_1 & \cdots & u_{n-1} & 0 & \ldots & 0\\ w_1 & \cdots & w_{n-1} & e_1 & \ldots & e_d\end{array}\right)  \right) < n \mbox{.}\end{equation}
\normalsize
Now (\ref{eq:thm_eq}) follows directly from the fact that the last $d$ rows of the matrix on the right are linearly independent and independent of the first $n-d$ rows.
\end{proof}

In other words, Theorem~\ref{thm:k_fix_rankVSsens0}  says that $g(x)=0$  has zero sensitivity in the variable $x_i$ with respect to $S$ at  $x^*$ if and only if the $i$-th column of the matrix $\frac{\partial g(x^*)}{\partial x}$ is essential for this matrix to have full rank.

\section{Local ACR}\label{sec:ACR}

We consider now robustness in another sense, by requiring that the $i$-th component of any solution to $g(x)=0$ only  attains a finite number of values. 
We  give two definitions: ACR and a local version, that we call  local ACR. We will see that this local property is closely related to zero sensitivity, easier to check than ACR and necessary   for ACR. Moreover, in some cases the two notions turn out to be equivalent  (see Proposition \ref{Prop:localACR-ACR} and Remark \ref{rk:param}). 
We provide the definitions in full generality in Subsection~\ref{subsec:def_lACR}, but study the case where $g(x)$ is a generalized polynomial with rational exponents in Subsection~\ref{subsec:detecting_lACR}.  In this setting, we develop practical tools to check for local ACR. In this section we present the results, but the proofs and details are provided in Section~\ref{sec:backgroundAG}, which relies on notions from algebraic geometry.

\subsection{Definition of local ACR}\label{subsec:def_lACR}
We now state the definition of ACR together with local ACR. The former was introduced in \cite{S-F} in the context of reaction networks for steady states of systems with mass-action kinetics. We keep the same name, even though we are concerned with general systems, not necessarily modeling concentrations.

\begin{Def}\label{def:localACR}
Let  $g\colon \Omega\rightarrow \R^s$ with $\Omega\subseteq \R^n$ be a function, 
  $\Vp$ be the set of solutions to the system $g(x)=0$ as in \eqref{eq:Vp}, and $\Vp'\subseteq \Vp$ be a non-empty subset.
\begin{itemize}
\item
The system $g(x)=0$ has \textit{Absolute Concentration Robustness (ACR)} with respect to $x_i$ over $\Vp'$ if  there exists $C\in\R$ such that  $x_i=C$ for all $x\in \Vp '$.
\item The system $g(x)=0$ has \textit{local Absolute Concentration Robustness (local ACR)} with respect to $x_i$ over $\Vp'$ if there exist $C_1,\dots,C_\ell \in \mathbb{R}$ such that $x_i\in \{C_1,\dots,C_\ell\}$ for all $x\in \Vp'$.
\end{itemize}
If $\Vp'=\Vp$, we simply say that the system has (local) ACR with respect to $x_i$.
\end{Def}

A system $g(x)=0$ has ACR with respect to $x_i$ if   $\Vp$   is contained in a hyperplane of the form $x_i-C=0$, and it has local ACR with respect to $x_i$ if $\Vp$ is contained in a finite union of disjoint hyperplanes with equations $x_i-C_1=0,\dots, x_i-C_\ell=0$ respectively.
Clearly, ACR implies local ACR. But the converse might not be true in general (see Example \ref{ex:lACR1} below). However, they are equivalent under certain conditions.

\begin{Prop}\label{Prop:localACR-ACR}
With the notation of Definition~\ref{def:localACR}, assume $\Vp'\neq \emptyset$. Then ACR over $\Vp'$ implies local ACR over $\Vp'$. Moreover, if $\Vp'$ is connected, then local ACR over $\Vp'$  is equivalent to ACR over $\Vp'$. 
\end{Prop}

\begin{proof}
The first statement is trivial. For the second statement,  by hypothesis, there exist  $C_1,\dots,C_\ell$ 
such that the projection of $\Vp'$ onto the variable $x_i$ is the set $\{C_1,\dots,C_\ell\}$. 
 As the projection of a connected set is connected and $\Vp'$ is connected by assumption, the set $\{C_1,\dots,C_\ell\}$ must be a singleton. 
\end{proof}

\begin{Rem}\label{rem:n=s_lACR} If $\Vp $ is finite, then the system $g(x)=0$ has local ACR in all variables. However, it may not have ACR in case the cardinality of $\Vp$ is larger than $1$.
\end{Rem}

\begin{Ex}\label{ex:introCRN}
Consider the following reaction network (c.f. Example \ref{ex:canonical})
\begin{align*}
X_1+X_2 & \ce{->[k_1]}  2X_2 & 
X_2 & \ce{->[k_2]}  X_1
\end{align*}
 with $k_1,k_2\in \mathbb{R}_{>0}$. With mass-action kinetics, it gives rise to the ODE system in $\Omega =\Rp ^2$
\begin{align*}
\tfrac{dx_1}{dt} & =-k_1x_1x_2+k_2x_2\mbox{,} &
\tfrac{dx_2}{dt} & =k_1x_1x_2-k_2x_2\mbox{.}
\end{align*} 
Let $f$ be defined by the right-hand side of the ODE system, and let $g(x)=k_1x_1x_2-k_2x_2$ obtained after removing linear dependencies of the entries of $f$. 
The set of positive steady states agrees with the zero set of $g$: 
\[\Vp =\{ x_1=\tfrac{k_2}{k_1},x_2>0\} \mbox{.}\] 
The system has ACR with respect to $x_1$, regardless of the specific values of $k_1,k_2$. This example was introduced in \cite{S-F} as a simple network having ACR. 
\end{Ex}

\begin{Ex}\label{ex:complex}
Consider the function $g=(g_1,g_2)$ in $\Omega =\Rp ^3$ defined by 
\begin{align*}
g_1(x) & =x_1^2x_3-x_2x_3-x_1^2+x_2\mbox{,}\\
g_2(x) & =x_1^4+2x_1^2x_2^2-2x_1^2+x_2^4-2x_2^2+x_3^2-2x_3+2\mbox{.}
\end{align*}
As $g_1(x)=(x_1^2-x_2)(x_3-1)$ and $g_2(x)=(x_1^2+x_2^2-1)^2+(x_3-1)^2$, the set of solutions to $g(x)=0$  is 
\[ \Vp =\{ x_3=1,x_1^2+x_2^2=1:x_1,x_2>0\}. \]
Therefore, the system has ACR with respect to $x_3$. 
\end{Ex}

\begin{Ex}\label{ex:lACR1}
Consider the following reaction network 
\begin{equation*}\label{eq:lACR1}
3X_1+X_2\ce{->[k_1]}  4X_1 \qquad
2X_1+X_2 \ce{->[k_2]}  3X_2  \qquad
X_1+X_2 \ce{->[k_3]}  2X_1,
\end{equation*}
  with $k_1,k_2,k_3\in \mathbb{R}_{>0}$ and
having two species ($n=2$) and 3 reactions. With mass-action kinetics, we obtain the ODE system
\begin{align*}
\tfrac{dx_1}{dt} & =k_1x_1^3x_2-2k_2x_1^2x_2+k_3x_1x_2\mbox{,}&
\tfrac{dx_2}{dt} & =-(k_1x_1^3x_2-2k_2x_1^2x_2+k_3x_1x_2)\mbox{.}
\end{align*}
As $\tfrac{dx_1}{dt}=-\tfrac{dx_2}{dt}$, we consider $g(x)=k_1x_1^3x_2-2k_2x_1^2x_2+k_3x_1x_2$. 
The positive steady states are the positive solutions to
\[ x_1x_2(k_1x_1^2-2k_2x_1+k_3)=0, \quad\textrm{equivalently}, \quad k_1x_1^2-2k_2x_1+k_3=0.\]
Let $D(k)=4(k_2^2-k_1k_3)$ be the discriminant of  $k_1x_1^2-2k_2x_1+k_3$. 
If $D(k)<0$, the system has no positive steady state. If $D(k)=0$, then  $\Vp$ consists of one double half-line $\left\{ x_1=\frac{k_2}{k_1}\in \Rp \right\} \cap \mathbb{R}^2_{>0}$, and the system has ACR with respect to $x_1$. 
If $D(k)>0$, then 
$\Vp$ consists of two half-lines, 
\begin{equation}\label{eq:disc}
\left\{ x_1=\tfrac{k_2}{k_1}-\tfrac{\sqrt{k_2^2-k_1k_3}}{k_1} \right\} \cap \mathbb{R}^2_{>0}\quad \mathrm{\; and \; }\quad\left\{ x_1=\tfrac{k_2}{k_1}+\tfrac{\sqrt{k_2^2-k_1k_3}}{k_1} \right\} \cap \mathbb{R}^2_{>0}\mbox{,}
\end{equation}
and the system has local ACR,  but not ACR, with respect to $x_1$.
 \end{Ex}

\begin{Rem} \label{rk:param}
If  $\Vp$ admits a continuous parametrization, that is, is the image of a continuous map from a connected subset of some $\R^\ell$, then $\Vp$ is connected. For any such system with $\Vp\neq \emptyset$, ACR is equivalent to local ACR by Proposition~\ref{Prop:localACR-ACR}.
In particular, this is the case when the system $g(x)=0$ is equivalent to a system consisting of binomials and   $\Vp\neq \emptyset$  (see \cite{Sturmfels-Binomial}). In the context of reaction networks, see \cite{CDSS} and \cite{P-MDSC} for such examples.
\end{Rem}

\subsection{Detecting local ACR}\label{subsec:detecting_lACR}

In this subsection we provide a criterion to determine local ACR via algebraic means. The results are derived employing concepts from algebraic geometry and hence applicable when $g$ is polynomial. Via an easy transformation, generalized polynomials with rational exponents can be treated in the same way. To see that, assume that $g$ consists of generalized polynomials with rational exponents:
\[ g=(g_1,\dots,g_s),\qquad g_i(x)= \sum_{\alpha\in A_i} c_\alpha x^\alpha,\quad i=1\dots,s,\]
with $A_i\subseteq \Q^{n}$ a finite set and $c_\alpha\in \R$ non-zero. 
 We have $g\in \mathcal{C}^1(\R^n_{>0},\R^s)$, so    $\Omega=\R^n_{>0}$, and we study positive solutions.
We establish  that any generalized polynomial $g$ can be transformed into a polynomial function $\widetilde{g}$ such that the respective sets of positive solutions are in one-to-one correspondence.
The idea is to make a change of variables such that the exponents of $g$ become integers, and to multiply the components of $g$ by a monomial such that all exponents become non-negative.

For every $j=1,\dots,n$, let $m_j$ be the least common multiple  of the denominators of the absolute value of all the exponents of $x_j$ in $g(x)$:
\[  m_j := {\rm lcm} \Big( \bigcup_{i=1}^s \{\textrm{denominator of }|\alpha_j|  : \alpha \in A_i \}\Big) \in \mathbb{Z}_{\geq 0}.\]  
Then for $i=1,\dots,s$, $g_i(z_1^{m_1},\ldots ,z_n^{m_n})$ is a generalized polynomial in $z=(z_1,\ldots ,z_n)$ with integer exponents in the set $\widetilde{A}_i:= \{ (\alpha_1m_1,\dots,\alpha_nm_n)\in \mathbb{Z}^n : \alpha \in A_i\}$.
For $j=1,\dots,n$, 
let 
\[ \beta(i)_{j} := {\rm min}\  \{ \delta \in \mathbb{Z}_{\geq 0} : \delta + \alpha_j \in \mathbb{Z}_{\geq 0} \textrm{ for all }\alpha\in \widetilde{A}_i\}.\]
With these definitions, $\widetilde{g}(z)$  defined by
\begin{equation}\label{eq:gtilde}
\widetilde{g}_i(z):= z^{\beta(i)} g_i(z_1^{m_1},\ldots ,z_n^{m_n})
\end{equation}
is a polynomial for $i=1,\dots,s$. Furthermore, the diffeomorphism
\begin{equation}\label{eq:corresp_x_z}
\varphi\colon \R^n_{>0}  \rightarrow  \R^n_{>0} \qquad 
z  \mapsto   (z_1^{m_1},\ldots ,z_n^{m_n}) 
\end{equation}
induces a homeomorphism between the set of positive solutions to $\widetilde{g}(z)=0$ and the set of positive solutions to $g(x)=0$ $$\Vp =\{ x\in \mathbb{R}^n_{>0}:g(x)=0 \} =\{ (z_1^{m_1},\ldots ,z_n^{m_n})\in \mathbb{R}^n_{>0}:\tilde{g}(z)=0 \} =\{ \varphi (z)\in \mathbb{R}^n_{>0}:\tilde{g}(z)=0 \} $$ (additionally, this map is locally a diffeomorphism between non-degenerate solutions), as $z^{\beta(i)}$ cannot vanish on $\R^n_{>0}$.

The following statements relating $g$ and $\widetilde{g}$  follow from the discussion above:
\begin{itemize}
\item The system $g(x)=0$ has local ACR with respect to $x_i$ over a set $\Vp'$ if and only if 
 the polynomial system $\widetilde{g}(z)=0$ has local ACR with respect to $z_i$ over the set $\varphi^{-1}(\Vp')$. 
\item By letting $z^\beta=(z^{\beta(1)},\dots,z^{\beta(n)})$, it holds that
\begin{equation}\label{eq:jacg}
\tfrac{\partial\widetilde{g}(z)}{\partial z} = g(\varphi(z))\tfrac{\partial z^{\beta }}{\partial z}  +   z^{\beta }\, \tfrac{\partial g(x)}{\partial x}|_{x=\varphi(z)}\, \tfrac{\partial \varphi(z)}{\partial z}.  
\end{equation}
If $z\in \R^n_{>0}$ is such that $\widetilde{g}(z)=0$, then the first summand vanishes.  As $\varphi$ is a diffeomorphism, the rank of $\tfrac{\partial\widetilde{g}(z)}{\partial z}$ agrees with the rank of 
$ \tfrac{\partial g(x)}{\partial x}|_{x=\varphi(z)}$ for any solution to $\tilde{g}(z)=0$. 
\item The system $g(x)=0$ has zero sensitivity  in $x_i$ with respect to $S$ at a point $x^*$ if and only if 
the system $\widetilde{g}(z)=0$ has zero sensitivity  in $z_i$ with respect to $S$ at the point $\varphi^{-1}(x^*)$.
\end{itemize}

Note that by construction, if $g$ is a polynomial, then $\widetilde{g}=g$.

\begin{Ex}\label{ex:lACR2r}
We consider the network in Example~\ref{ex:lACR1} with a power-law kinetics   \eqref{eq:powerlaw} with exponent matrix
\[B=\begin{pmatrix} 1 & \sfrac{2}{3} & \sfrac{-1}{3} \\ \sfrac{2}{3}  & \sfrac{2}{3}  & \sfrac{2}{3} \end{pmatrix}\mbox{.}\]
With $k_1=k_2=k_3=1$, the generalized polynomial describing $\tfrac{dx_1}{dt}$ is: 
\[ g(x)=x_1x_2^{\sfrac{2}{3}}-2x_1^{\sfrac{2}{3}}x_2^{\sfrac{2}{3}}+x_1^{-\sfrac{1}{3}}x_2^{\sfrac{2}{3}}.\]
With the notation above, we have $m_1=m_2=3$ and
\[ g(z_1^3,z_2^3)=z_1^3 z_2^2 -2z_1^2 z_2^2+ z_1^{-1} z_2^{2}. \]
This yields to $\beta(1)=(1,0)$ and hence, the polynomial $\widetilde{g}$ from \eqref{eq:gtilde} is
\begin{equation*}\label{eq:gtildex}
\widetilde{g}(z_1,z_2)=  z_1 g(z_1^3,z_2^3) =  
z_1^4 z_2^2 -2z_1^3 z_2^2+  z_2^{2}. 
\end{equation*}
The set of positive solutions to $\widetilde{g}(z)=0$ is in one-to-one correspondence with
 the set of positive solutions to $g(x)=0$   via the diffeomorphism $(z_1,z_2)\longmapsto (z_1^3,z_2^3)$.
\end{Ex}

We have established that the study of the zero set of generalized polynomials 
 can be reduced to the study of the zero set of polynomials. Even though our interest is on the positive real solutions to the system $g(x)=0$,  complex solutions  are required in order to use certain algebraic-geometric tools.   In this section, a basic concept is required, namely that of \emph{irreducible components}.

Let $g=(g_1,\dots,g_s)$ be a polynomial function in $n$ variables with linearly independent entries.
The set of complex solutions to $g(x)=0$ is a complex algebraic variety, that is,  the zero set of a collection of polynomials, denoted by $\mathbb{V}_{\mathbb{C}}(\langle g_1,\ldots ,g_s\rangle )$. 
Any algebraic variety can be written in a unique  way as the union of finitely many \textit{irreducible components}  $\mathbb{V}_{\mathbb{C}}(\langle g_1,\ldots ,g_s\rangle )=\mathcal{Y} _1\cup \ldots \cup \mathcal{Y}_m$, where $\mathcal{Y}_j\nsubseteq \mathcal{Y}_k$ whenever $j\neq k$, and each $\mathcal{Y}_j$ is  a (non-empty) algebraic variety that cannot be  decomposed into two  algebraic varieties in the same way. We refer to Section \ref{subsec:irred_deg} for details.

We will refer to the \textit{irreducible components of $\Vp =\mathbb{V}_{\mathbb{C}}(\langle g_1,\ldots ,g_s\rangle  ) \cap \mathbb{R}_{>0}^n$}, meaning the subsets of the form $\mathcal{Y}\cap \mathbb{R}_{>0}^{n}\subseteq \Vp $ where $\mathcal{Y}$ is an irreducible component of   $\mathbb{V}_{\mathbb{C}}(\langle g_1,\ldots ,g_s\rangle  )$. More generally, if $g$ is a generalized polynomial, the irreducible components of $\Vp$ are by definition the sets obtained by applying the map $\varphi$ in  \eqref{eq:corresp_x_z} to the irreducible components of $\mathbb{V}_{\mathbb{C}}(\langle \widetilde{g}_1,\ldots ,\widetilde{g}_s\rangle  ) \cap \mathbb{R}_{>0}^n$ for the polynomial map in \eqref{eq:gtilde}.

We will sometimes restrict our results to a subset $\mathcal{C}\subseteq \{ \mathcal{Y}_1,\ldots ,\mathcal{Y}_m\} $ of the irreducible components of $\Vp$. We let   
\begin{equation}\label{eq:Vc}
\Vp _{\mathcal{C}}=\left(\cup _{\mathcal{Y}\in \mathcal{C}} \mathcal{Y}\right) \cap \mathbb{R}_{>0}^n
\end{equation} be the part of $\Vp$ that lies on the selected components. In particular, we are mostly interested in the set $\Snd(g)$ of irreducible components that have a positive non-degenerate point: 
\begin{equation}\label{eq:Vnd}
\Vnd(g):= \Vp _{\Snd(g)}.
\end{equation} 
The existence of a non-degenerate point in each considered irreducible component is a main hypothesis in our criteria.

\begin{Ex}\label{ex:introCRN2}
Back to Example \ref{ex:introCRN}, 
$\mathbb{V}_{\mathbb{C}}(\langle g\rangle )$ is
$$\mathbb{V}_{\mathbb{C}}(\langle k_1x_1x_2-k_2x_2\rangle )=\mathbb{V}_{\mathbb{C}}(\langle x_2\rangle )\cup \mathbb{V}_{\mathbb{C}}(\langle k_1x_1-k_2\rangle )\mbox{,}$$
which has two irreducible components. Only the second component has points in $\R^2_{>0}$ and hence $\Vp =\mathbb{V}_{\mathbb{C}}(\langle k_1x_1x_2-k_2x_2\rangle )\cap \mathbb{R}^2_{>0}=\mathbb{V}_{\mathbb{C}}(\langle k_1x_1-k_2\rangle )\cap \mathbb{R}^2_{>0} $. Note that the system has local ACR and ACR with respect to $x_1$ over $\Vp $.
\end{Ex}

It turns out that local ACR over $\Vp$ corresponds to ACR over each irreducible component under certain non-degeneracy conditions. 
The following result is shown in Section~\ref{sec:backgroundAG}. Specifically, it follows from Corollary \ref{cor:alg_lACR} and Proposition~\ref{prop:degVSsing},  together with the correspondence between generalized polynomials and polynomials above.

\begin{Prop}\label{prop:alg_lACR_4}
Let $g=(g_1,\dots,g_s)$ be a generalized polynomial function in $\R^n_{>0}$, and let $\mathcal{C}=\{ \mathcal{Y}_1, \ldots , \mathcal{Y}_k\} \subseteq \Snd(g)$ be non-empty.
The system  $g(x)=0$ has local ACR  with respect to $x_i$ over $\Vp _{\mathcal{C}}$ if and only if it has ACR  with respect to $x_i$ over $\Vp _{\mathcal{Y}_j}$ for all $j=1,\ldots ,k$.
 \end{Prop}

In the context of Proposition~\ref{prop:alg_lACR_4}, the values that $x_i$ attains in each  irreducible component need not be different.  
We illustrate now Proposition~\ref{prop:alg_lACR_4} with the examples in Section \ref{subsec:def_lACR}. 

All points of $\Vp$ in  Example \ref{ex:complex} are degenerate, so Proposition \ref{prop:alg_lACR_4} (or Theorem~\ref{prop:6regb_4} below) cannot be applied.
For Example \ref{ex:lACR1}, $\Vp$ has non-degenerate  points only for $D(k)>0$. In this case
  $\mathbb{V}_{\mathbb{C}}(\langle x_1x_2(k_1x_1^2-2k_2x_1+k_3)\rangle )$ decomposes as
$$\mathbb{V}_{\mathbb{C}}\big( \langle x_1 \rangle \big) \cup \mathbb{V}_{\mathbb{C}}\big( \langle x_2\rangle \big) \cup \mathbb{V}_{\mathbb{C}}\Big( \langle x_1-\tfrac{k_2}{k_1}+\tfrac{\sqrt{k_2^2-k_1k_3}}{k_1}\rangle \Big) \cup \mathbb{V}_{\mathbb{C}}\Big( \langle x_1-\tfrac{k_2}{k_1}-\tfrac{\sqrt{k_2^2-k_1k_3}}{k_1}\rangle \Big) \mbox{.}$$
Only the last two components have positive points, hence $\Vp$ has two irreducible components, which in fact consist of non-degenerate points. Each of these components, when considered separately, shows ACR with respect to $x_1$, hence the system has local ACR with respect to $x_1$ over $\Vp $ by Proposition~\ref{prop:alg_lACR_4}.
Similarly, the irreducible components of $g$ in Example~\ref{ex:lACR2r} are by definition obtained by applying $\varphi$ to the irreducible components of $\widetilde{g}$  in \eqref{eq:gtildex}.
Proceeding as above, we find that $\mathbb{V}_{\mathbb{C}}(\langle z_1^4 z_2^2 -2z_1^3 z_2^2+  z_2^{2}\rangle )\cap \R^2_{>0}$ has two irreducible components given by the two positive roots of $z_1^4  - 2z_1^3  +  1$. As each component is included in a hyperplane $z_1-C=0$, both $g$ and $\widetilde{g}$ have local ACR by Proposition~\ref{prop:alg_lACR_4}.

\smallskip
These examples are so small that the decomposition of $\Vp$ into irreducible components can be found, but in general this computation is unfeasible. However, the algebraic implications of Proposition \ref{prop:alg_lACR_4}, developed along Section \ref{subsec:irred_deg}, allow us to state the following more practical criterion for detecting local ACR. Its proof can be found in Section~\ref{subsec:criterion_lACR}.

 \begin{Thm}\label{prop:6regb_4} 
Let $g=(g_1,\dots,g_s)$ be a generalized polynomial function in $\R^n_{>0}$ and let $\mathcal{C}\subseteq \Snd(g)$ such that  $\Vp _{\mathcal{C}}\neq \emptyset$. The system $g(x)=0$ has local ACR with respect to $x_i$ over $\Vp _{\mathcal{C}}$ if and only if 
 \begin{equation}\label{eq:rank_dim2_4}
 \rank\Big( \big( \tfrac{\partial g(x^*)}{\partial x}\big)^i \Big)<  s, \qquad \textrm{for all }x^*\in \Vp _{\mathcal{C}}.
  \end{equation}   
\end{Thm}

  \smallskip
To better understand why non-degenerate points are required, see Remark \ref{rem:ex_Deg}.
In Example \ref{ex:introCRN}, we have $  \tfrac{\partial g(x)}{\partial x} = (-k_1x_2 \quad -k_1x_1+k_2) $, and $\big( \tfrac{\partial g(x)}{\partial x}\big)^1=\big( -k_1x_1+k_2\big)$   has rank $0$ at all points of $\Vp$, where $x_1=\tfrac{k_2}{k_1}$.
In Example \ref{ex:lACR1}, as $\big( \tfrac{\partial g(x)}{\partial x}\big)^1=\big( k_1x_1^3-2k_1^2+k_3x_1\big) $ has rank $0$ for  $x\in\Vp$, see \eqref{eq:disc}, $g(x)=0$ has local ACR with respect to $x_1$.

\begin{Rem}
Theorem~\ref{prop:6regb_4} gives a criterion for local ACR over $\Vnd(g)$. 
Note that 
\begin{center}
local ACR with respect to $x_i$  \quad $\Rightarrow$ \quad local ACR with respect to $x_i$ over $\Vnd(g)$.
\end{center}
If $\Vnd(g)\neq \emptyset$, then Theorem~\ref{prop:6regb_4} gives a necessary condition for local ACR over $\Vp$. 
If the system  has local ACR with respect to $x_i$ over  $\Vnd (g)$  but not over $\Vp$, then it only fails to have local ACR on the irreducible components of $\Vp$  consisting entirely of degenerate points.
\end{Rem}

\begin{Rem}
If $g=(g_1,\dots,g_n)$ is a generalized polynomial in $\R^n_{>0}$, that is $s=n$,  and $x^*\in \Vp$, then \eqref{eq:rank_dim2_4} holds trivially, as 
  $\tfrac{\partial g(x^*)}{\partial x}$ is an $(n\times n)$-matrix. For $\mathcal{C}\subseteq \Snd(g)$ as in Theorem~\ref{prop:6regb_4}, we have that $\Vp _{\mathcal{C}}$ is finite.  This is in accordance with the observation in Remark~\ref{rem:n=s_lACR}.
\end{Rem}

\section{Local ACR and zero sensitivity}\label{subsec:lACR_via_sens}
Assume that $g$ is a vector of  generalized  polynomials and $\Omega=\R^n_{>0}$ as in the previous section. 
Zero sensitivity refers to a property around every solution to $g(x)=0$, and local ACR to 
the possible values of an entry of the solution set.
Comparing  Theorem~\ref{prop:6regb_4}  with Theorem~\ref{thm:k_fix_rankVSsens0}, both local ACR and zero sensitivity require the same condition on the rank of $g$, but apply to   different sets of points: for local ACR, the condition applies to irreducible components containing a point in $\NDeg(g)$, while for zero sensitivity to points in $\NDeg_S(g)$.    Following previous notation, we denote the set of  irreducible components of $\Vp$ that intersect $\NDeg_S(g)$  by $\SndS$, and let 
\begin{equation}\label{eq:VndS}
\VndS:= \Vp _{\SndS}\subset \Vp.
\end{equation} Recall that $ \SndS\subseteq \Snd(g)$.

Theorem \ref{thm:6_5} below formalizes the relation between local ACR and zero sensitivity for a fixed generalized polynomial system $g(x)=0$. This will in turn derive into a practical criterion for local ACR in Theorem \ref{thm:main_b}, for parametric families $g_{k}$ as in \eqref{eq:powerlaw} arising for example from the study of reaction networks.

\begin{Thm}\label{thm:6_5} 
Let $g=(g_1,\dots,g_s)$ be a generalized polynomial function in $\R^n_{>0}$, $S\subseteq \R^n$ a vector subspace of dimension $s$, and $\mathcal{C}\subseteq \SndS$ be non-empty.
The following are equivalent:
\begin{enumerate}
	\item The system $g(x)=0$ has local ACR with respect to $x_i$ over $\Vp _{\mathcal{C}}$.
	\item The system $g(x)=0$ has zero sensitivity  in $x_i$ with respect to $S$ for all $x^*\in \Vp _{\mathcal{C}}\cap \NDeg_S(g)$. 
	\item 
	$\rank\Big( \big( \tfrac{\partial g(x^*)}{\partial x}\big)^i \Big)<  s$ for all $x^*\in \Vp _{\mathcal{C}}$.
\end{enumerate}

If $\mathcal{C}\subseteq \Snd(g)$, then (1) $\Leftrightarrow$ (3) $\Rightarrow$ (2) still holds. 
\end{Thm}

\begin{proof}
The equivalence (1) $\Leftrightarrow$ (3)  follows from Theorem~\ref{prop:6regb_4}, as $ \SndS\subseteq \Snd(g)$. The equivalence (3) $\Rightarrow$ (2) holds by Theorem~\ref{thm:k_fix_rankVSsens0}. For the reverse implication, if $g$ is polynomial, 
Theorem~\ref{thm:k_fix_rankVSsens0}  gives that the condition in (3) holds for all points in $\Vp _{\mathcal{C}}\cap \NDeg_S(g)$. 
By  Remark~\ref{rem:Zar_dense}, as  $\NDeg_S(g)$ is dense in $\Vp _{\mathcal{C}}$ by Remark~\ref{rem:Vclosures}, (3) holds as well. If $g$ is not polynomial, then we consider $\widetilde{g}$ from \eqref{eq:gtilde} and the relation to $g$ described around Equation \eqref{eq:jacg}.
\end{proof}

Recall that zero sensitivity is only defined for points in $\NDeg_S(g)$. Hence even if (2) holds for $\mathcal{C}=\SndS$, it could still be the case that (3) does not hold on $\Vnd(g)$ as it could fail for an irreducible component $\mathcal{Y}\in \Snd(g) \setminus \SndS$. 

Observe that if $g$ is not polynomial, statement Theorem \ref{thm:6_5}(3) can be indifferently checked with $g$ or $\widetilde{g}$ in \eqref{eq:jacg}. 
When $g$ is polynomial, we give in Proposition~\ref{thm:6} a criterion, testable using computational algebra software, which is equivalent to  Theorem \ref{thm:6_5}(3).    However, the computational cost will easily make the criterion impractical in the application where  $g$ is parametric. 

When considering a family $g_k$ of the form \eqref{eq:powerlaw}, then Theorem~\ref{thm:6_5} reduces to a simple criterion for zero sensitivity and local ACR for the whole family. 
Specifically, let
$N\in \R^{s\times n}$ of full rank $s$, $B\in \mathbb{Q}_{\geq 0}^{n\times r}$ and $k\in \Rp ^r$, and consider $g_k\colon \R^n\rightarrow \R^s$ defined by
\begin{equation}\label{eq:powerlaw_k}
g_k(x)= N \diag(k) x^B, \qquad x\in  \R^n_{>0}.
\end{equation}
Let $\Vp^k$ be the set of solutions to $g_k(x)=0$ in $\R^n_{>0}$.
 
The following proposition gathers useful results on the family of Jacobian matrices $\tfrac{\partial g_k(x^*)}{\partial x}$ for all $k\in \R^r_{>0}$ and $x\in \Vp^k$, that bypass the problem of explicitly finding  $\Vp^k$. It is based on so-called \emph{convex parameters}, that go back to Clarke \cite{Clarke:1980tz}.

\begin{Prop}\label{prop:checkingrank}
Let $g_k(x)=N \diag(k) x^B$ be as in \eqref{eq:powerlaw_k} and assume $\ker(N)\cap \R^r_{>0}\neq \emptyset$.
\begin{enumerate}[(i)] 
\item  The following statements are equivalent: 
\begin{enumerate}[(1)]
\item $\rank  \big( \tfrac{\partial g_k(x^*)}{\partial x}\big) =s$ for all $k\in \R^r_{>0}$ and  $x^*\in \Vp^k$ (hence $\Vp^k=\Vnd(g_k)$).
\item For all $v\in \mathrm{ker}(N)\cap \R^r_{>0}$, at least one of the ($s\times s$)-minors of the matrix $N\mathrm{diag}(v)B^t$ is different from zero.
\end{enumerate} 

\smallskip
\item The following statements are equivalent for $i\in \{1,\dots,n\}$:
\begin{enumerate}[(1)]
\item $\rank\Big( \big( \tfrac{\partial g_k(x^*)}{\partial x}\big)^i \Big)<  s$ for all $k\in \R^r_{>0}$ and  $x^*\in \Vp^k$.
\item All ($s\times s$)-minors of the matrix $(N\mathrm{diag}(v)B^t)^i$ are zero for all $v\in \mathrm{ker}(N)$.
\end{enumerate}

\smallskip
\item 
Let $S\subseteq \R^n$ be a vector subspace of dimension $s$ and $W$ any matrix of maximal rank such that $S=\ker(W)$. If (ii,1) holds, then for all $v\in \ker(N)$,  $h_i$ divides the polynomial 
\[ p_v(h):=\mathrm{det}\left( \begin{array}{c}N\mathrm{diag}(v)B^t\mathrm{diag}(h)\\ W\end{array}\right)\in \R[h_1,\dots,h_n].\]
\end{enumerate}
\end{Prop}
\begin{proof}
By the form of $g_k$, for $x\in \R^n_{>0}$, we have
\begin{equation*}\label{eq:convex}
\tfrac{\partial g_k(x)}{\partial x} = N \diag( \diag(k)x^B ) B^t \diag(\tfrac{1}{x}),
\end{equation*}
where $\tfrac{1}{x}$ is taken component-wise.
As all entries of $\tfrac{1}{x}$ are positive, the rank of $\frac{\partial g_k(x)}{\partial x}$,  resp. $\big( \frac{\partial g_k(x)}{\partial x} \big)^i$  agrees with the rank of 
$(N \diag(\diag(k) x^B )B^t)$, resp. $(N \diag(\diag(k) x^B )B^t)^i$,
since multiplication by $\mathrm{diag}(\tfrac{1}{x})$ only scales the minors by a positive real number. 
Now  observe that there is an equality of sets
\begin{equation*}\label{eq:Nker}
\{ \diag(k)x^B \mid k\in \R^r_{>0}, x \in \Vp^k \} = \{ v \in \ker(N)\cap \R^r_{>0}\}.
\end{equation*}
The inclusion $\subseteq $ is clear, as $0=g_k(x) = N \diag(k) x^B$ gives that $\diag(k)x^B\in \ker(N)\cap \R^r_{>0}$. 
For the reverse inclusion, given $v\in \ker(N)\cap \R^r_{>0}$, let $k=v$ and $x$ with all entries equal to  $1$.

From this discussion, statement (i) follows, as 
\[\{ \rank  \big(\tfrac{\partial g_k(x^*)}{\partial x}\big): k\in \R^r_{>0}, x^*\in \Vp^k\} = \{ \rank  \big(N\mathrm{diag}(v)B^t\big) : v\in \ker(N)\cap \R^r_{>0}\}.\]
For statement (ii), we have analogously that (1) holds if and only if $\rank((N\mathrm{diag}(v)B^t)^i)<s$ for all $v=(v_1,\ldots ,v_r)\in \ker(N)\cap \R^r_{>0}$.
Let $G(v)$ be any minor of size $s$ of $(N\mathrm{diag}(v)B^t)^i$, and  consider it as a polynomial in $v_1,\dots,v_r$.   As $\mathrm{ker}(N)$ is a vector subspace of $\mathbb{R}^r$, hence an algebraic variety, and is the smallest algebraic variety containing $\mathrm{ker}(N)\cap \Rp ^r$, which is non-empty, 
  the polynomial $G(v)$ vanishes on $\ker(N)\cap \R^r_{>0}$ if and only if it does on $\ker(N)$ (see Remark \ref{rem:Zar_dense}). This shows that (1) is equivalent to (2) and concludes the proof of (ii). 

\smallskip
To show (iii), we introduce the following notation for the cofactors of the relevant matrix: 
{\small \[ \Delta _j^i=(-1)^{i+j}\mathrm{det}\left( \left( \begin{array}{c}N\mathrm{diag}(v)B^t\mathrm{diag}(h)\\ W\end{array}\right)_j^i\right).
\] }%
Let $u_i$ denote the $i$-th column of  $N\mathrm{diag}(v)B^t\mathrm{diag}(h)$ and $w_i$   the $i$-th column of $W$. By expanding the determinant along the $i$-th column, we have
\[ p_v(h) = \mathrm{det}\left( \begin{array}{c}N\mathrm{diag}(v)B^t\mathrm{diag}(h)\\ W\end{array}\right) =\left( \Delta _{1}^i,\ldots ,\Delta _{n-d}^i\right) u_i+\left( \Delta _{n-d+1}^i,\ldots ,\Delta _{n}^i\right)  w_i. \]

 Note that $u_i=h_i\tilde{u}_i$, where $\tilde{u}_i$ is the $i$-th column of $N \mathrm{diag}(v)B^t$. Furthermore, observe that for $j=n-d+1,\ldots ,n$,   $\Delta _{j}^i$  is a linear combination of the minors of $N\mathrm{diag}(v)B^t$ not involving column $i$. 
 It follows from the equivalence (1)   $\Leftrightarrow $   (2) that all minors of $N\mathrm{diag}(v)B^t$ not involving column $i$ vanish, 
and hence so do $\Delta _{n-d+1}^i,\ldots ,\Delta _n^i$. Therefore,
\[p_v(h) = h_i\left( \Delta _{1}^i,\ldots ,\Delta _{n-d}^i\right) \tilde{u}_i\mbox{,}\] 
and $h_i$ divides the polynomial in the statement. This concludes the proof of the proposition.
\end{proof}

In order to apply Proposition~\ref{prop:checkingrank}(ii) and (iii) in practice, we parametrize $\ker(N)$ by first finding a basis. Then the relevant minors of $N\mathrm{diag}(v)B^t$ and $p_v(h)$  are polynomials in the new parameters (and $h$), and the conditions can be readily verified. In particular, Proposition~\ref{prop:checkingrank}(ii)(2) holds if all relevant minors are identically zero as polynomials. For Proposition~\ref{prop:checkingrank}(i), we consider minimal generators of the cone 
$\ker(N)\cap \R^r_{>0}$ (cf. \cite{conradi-feliu-mincheva}). See Example \ref{ex:conv_par} for an example.

Combining Proposition~\ref{prop:checkingrank} and Theorem \ref{thm:6_5}, we obtain the following  practical criterion for testing local ACR on the whole parametric family $g_k(x)$.

\begin{Thm}\label{thm:main_b}
Let $g_k(x) = N \diag(k) x^B$ be as in \eqref{eq:powerlaw_k}  and assume $\ker(N)\cap \R^r_{>0}\neq \emptyset$.
Let $S\subseteq \R^n$ be a vector subspace of dimension $s$. Consider the following statements: 
\begin{enumerate}
\item  The system $g_k(x)=0$  has local ACR with respect to $x_i$ over $\Vnd(g_k)$, for all $k\in \Rp ^r$ such that  $\Vnd(g_k)\neq \emptyset$.
\item The system $g_k(x)=0$  has zero sensitivity in $x_i$ with respect to $S$ for all $x^*\in  \NDeg_S(g_k)$ and $k\in \Rp ^r$.
\item All ($s\times s$)-minors of the matrix $(N\mathrm{diag}(v)B^t)^i$ are zero for all $v\in \mathrm{ker}(N)$.
\item For all $v\in \mathrm{ker}(N)\cap \R^r_{>0}$, at least one of the ($s\times s$)-minors of the matrix $N\mathrm{diag}(v)B^t$ is different from zero.
\end{enumerate}
Then  
\begin{itemize}
\item
(1) $\Leftrightarrow$ (3) $\Rightarrow$ (2). 
\item If additionally $\Vnd(g_k)_S=\Vnd(g_k)$ for all $k\in \R^r_{>0}$, then 
(1) $\Leftrightarrow$ (3) $\Leftrightarrow$ (2).
\item If (4) holds, then $\Vp^k=\Vnd(g_k)$, and hence (3) gives a criterion for local ACR over $\Vp^k$.
\end{itemize}
 \end{Thm}

 Theorem~\ref{thm:main_b}(3) gives a criterion for local ACR based on the computation of the 
 determinant of a collection of symbolic $(s\times s)$ matrices. 
The computational cost of this is much lower than finding the irreducible components of $\Vp$, which is  impossible for realistic systems. As this type of system arises in the context of chemical reaction networks, where the search for ACR is a relevant and difficult question, we have provided  an easy-to-check necessary condition to address ACR.
 Additionally, in view of Proposition~\ref{prop:checkingrank}, 
  a necessary condition for local ACR  (and hence ACR) under the hypothesis that $\Vnd(g_k)=\Vp^k$, is that  $h_i$ divides the polynomial 
\begin{equation}\label{eq:detv}\mathrm{det}\left( \begin{array}{c}N\mathrm{diag}(v)B^t\mathrm{diag}(h)\\ W\end{array}\right)\in \R[h_1,\dots,h_n],
\end{equation}
for all $v\in \ker(N)$ and \emph{any choice} of matrix $W\in \R^{d\times n}$. If the rows of $W$ and $N$ are not linearly independent, then the criterion holds as well, but is not informative.   
This necessary condition requires the computation of only one determinant of a symbolic matrix and hence provides a strategy to routinely scan parametrized systems of the form \eqref{eq:powerlaw_k} for the existence of ACR or local ACR.

Theorem~\ref{thm:main_b}(4) gives a sufficient criterion for $\Vnd(g_k)=\Vp^k$ to hold. 
Note that this condition holds for all  \emph{injective networks.} These are namely networks where 
the polynomial in \eqref{eq:detv}, now  seen as a polynomial in $\R[v_1,\dots,v_r,h_1,\dots,h_n]$, 
is non-zero and has all coefficients of the same sign  \cite{MullerSigns,W-F_SIAM}. 
Then the polynomial does not vanish at any positive value of the variables, and in particular $N\mathrm{diag}(v)B^t$ has rank $s$ for all $v\in \R^r_{>0}$. From this  Theorem~\ref{thm:main_b}(4) follows. Injective networks are of interest as the existence of multiple steady states is precluded. As finding the rank of $N\mathrm{diag}(v)B^t$
for any $v\in \R^r_{>0}$ does not require the computation of the cone
 $\ker(N)\cap \R^r_{>0}$, it is a good idea to first decide whether a maximal minor of $N\mathrm{diag}(v)B^t$ is a non-zero polynomial with all coefficients of the same sign. See Example~\ref{ex:idh} below.

  \begin{Ex}\label{ex:lACR2}
Example~\ref{ex:lACR1} corresponds to the matrices
{\small \[  N=\begin{pmatrix}
1 & -2 & 1 
\end{pmatrix}, \qquad B=\begin{pmatrix} 3 & 2 & 1 \\ 1 & 1 & 1
\end{pmatrix}. \]}
With $v=(-a+2c,c,a)$ parametrizing $\ker(N)$, we have  $N\mathrm{diag}(v)B^t=( -2a+2c \quad 0)$.
Clearly, removal of the first column causes the rank to drop. Hence, this example has local ACR with respect to $x_1$ over $\Vnd(g_k)$, as already noticed in Example~\ref{ex:lACR1}.

In Example~\ref{ex:lACR2r}, we modified the exponent matrix $B$ of Example~\ref{ex:lACR1}. More generally, by considering an arbitrary rational exponent matrix $B=(b_{ij})\in \Q^{2\times 3}$ we have
{\small 
\begin{align*} N\mathrm{diag}(v)B^t&=\begin{pmatrix} 1 & -2 & 1 \end{pmatrix} \begin{pmatrix} -a+2c & 0 & 0 \\ 0 & c & 0 \\ 0 & 0 & a \end{pmatrix} \begin{pmatrix} b_{11} & b_{21} \\ b_{12} & b_{22} \\ b_{13} & b_{23} \end{pmatrix} \\[5pt]
&= \Big( (b_{13}-b_{11})a + (2b_{11}- 2b_{12})c \qquad (b_{23}-b_{21})a + (2b_{21}- 2b_{22})c \Big) \mbox{.}\end{align*}}
The rank of $\left( N\mathrm{diag}(v)B^t\right)^i$ drops for all $a,c$ if and only if $b_{j1}=b_{j2}=b_{j3}$ with $j=3-i$, in which case the system has local ACR with respect to $x_i$ over $\Vnd(g_k)$ by Theorem \ref{thm:main_b}. 
Example~\ref{ex:lACR1} and Example~\ref{ex:lACR2r} satisfy the first relation. In this example, 
Theorem \ref{thm:main_b}(4) does not hold as expected, as we saw in Example~\ref{ex:lACR1} that for some values of $k$ there is one irreducible component consisting only of degenerate points.

\end{Ex}

\begin{Ex}\label{ex:idh}
Consider the core ACR module of the  IDHKP-IDH system in \emph{E. coli} considered in the seminal paper on ACR \cite[Fig. 3]{S-F}:
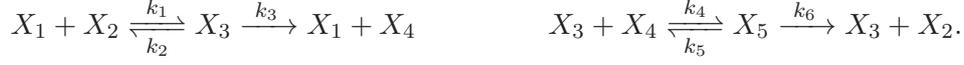
\begin{align*}
X_1 + X_2  & \ce{<=>[k_1][k_2]} X_3 \ce{->[k_3]} X_1 + X_4  & X_3 + X_4  & \ce{<=>[k_4][k_5]} X_5 \ce{->[k_6]} X_3 + X_2. 
\end{align*}
With mass-action kinetics, the system has ACR with respect to $x_4$, as shown in \cite{S-F}.
Let us consider power-law kinetics with exponent matrix having rational entries of the form 
 {\small \[  B=\begin{pmatrix} b_{11} & 0 & 0 & 0 & 0 & 0\\ b_{21} & 0 & 0 & 0 & 0 & 0\\ 0 & b_{32} & b_{33} & b_{34} & 0 & 0\\ 0 & 0 & 0 & b_{44} & 0 & 0\\ 0 & 0 & 0 & 0 & b_{55} & b_{56}\\
\end{pmatrix}, \qquad b_{ij}>0.\]}%
Mass-action kinetics corresponds to the case where all non-zero entries of $B$ are equal to $1$.
The stoichiometric matrix has $5$ rows and rank $3$. After removing linear dependencies, the coefficient matrix of $g_k(x)$ is 
{\small \[ N= \begin{pmatrix}
-1 & 1 & 0 & 0 & 0 & 1 \\ 0 & 0 & 1 & -1 & 1 & 0 \\ 0& 0 & 0 & 1 & -1 & -1  
\end{pmatrix} .   \]}

We verify Theorem \ref{thm:main_b}(4). 
We consider any matrix of the form $N\mathrm{diag}(v)B^t$ for $v\in \R^6_{>0}$:
{\small 
\begin{align*}
N\mathrm{diag}(v)B^t & = 
  \left( \begin{array}{ccccc}   -v_1  b_{{11}}&  -v_1 b_{{21}}& v_2 b_{{32}}&0& v_6b_{{56}}\\  
0&0&v_3b_{{33}} - v_4  b_{{34}}&  -v_4
  b_{{44}}&v_5b_{{55}}\\0&0&  v_4 b_{{34}}& v_4 b_{{44}}
  &-v_5b_{{55}}-v_6b_{{56}}\end {array} \right).
\end{align*}
}%
The minor given by the second, third and fourth columns is $ -v_1v_3v_4 b_{21}b_{33} b_{44}$,
which does not vanish for any $v\in \R^6_{>0}$ and $b_{ij}\in \R_{>0}$.
Therefore $\Vp^k=\Vnd(g_k)$ for all $k\in \R^6_{>0}$, that is, all irreducible components of $\Vp^k$ contain only non-degenerate points. Note that we did not need to impose $v\in \ker(N)\cap \mathbb{R}^r_{>0}$ in this case, as the  considered minor does not vanish anywhere in $\mathbb{R}^r_{>0}$.

We now apply Theorem \ref{thm:main_b}(3). To this end,  we consider the basis 
$w_1=(0,0,0,1,1,0),w_2=(1,1,0,0,0,0), w_3=(1,0,1,1,0,1)$ of $\ker(N)$, such that any vector in 
 $\ker(N)$ is of the form $v=a_1 w_1 + a_2w_2+a_3w_3$ with $a_1,a_2,a_3\in \R$. 
Substituting this into $N\mathrm{diag}(v)B^t$, we find that the rank of $N\mathrm{diag}(v)B^t$ drops after removing the fourth column, if and only if  
\begin{equation}\label{eq:b}
b_{33}=b_{34},\quad\textrm{and}\quad b_{55}=b_{56}\mbox{.}
\end{equation}
Hence, by Theorem \ref{thm:main_b}(1), when \eqref{eq:b} holds,  the system has local ACR with respect to $x_4$. As expected, mass-action kinetics fulfils \eqref{eq:b}.  If \eqref{eq:b} does not hold, then the system does not have local ACR with respect to any variable. 

\end{Ex}

\begin{Ex}\label{ex:conv_par}
We consider the family of functions $g_k$ as in \eqref{eq:powerlaw_k} defined by the matrices
{\small \begin{equation*}
N=\left( \begin{array}{cccc}
-1 & 1 & 1 & 0\\
0 & 0 & 1 & -1 
\end{array} \right) \mbox{,\; }
B=\left( \begin{array}{cccc}
1 & 1 & 0 & 1 \\
0 & 0 & 1 & 1 \\
 2& 2 & 0 & 2
\end{array} \right),
\end{equation*}}%
giving
\begin{equation*}\label{eq:conv_par_system}
g_{k,1}(x)=-k_1 x_1x_3^2+k_2x_1x_3^2+k_3x_2,\qquad 
g_{k,2}(x)=k_3x_2-k_4 x_1x_2x_3^2
 \mbox{.}\end{equation*}

We check first that $\Vp^k=\Vnd(g_k)$ holds using Theorem \ref{thm:main_b}(4).
For $v\in \R^4$, we have 
\begin{equation}\label{eq:conv_par_matrix}
N\mathrm{diag}(v)B^t = 
  \left( \begin{array}{ccccc}   v_2-v_1 & v_3 & 2v_2-2v_1 \\ -v_4 & v_3-v_4 & -2v_4 \end {array} \right)
\end{equation}
For $v=(1,2, \tfrac{1}{2},1)$, we have \[ N\mathrm{diag}(v)B^t =\begin{pmatrix} 1 & \tfrac{1}{2} & 2 \\ 
-1 & -\tfrac{1}{2} & -2 \end{pmatrix}, \] which has rank $1$.
Hence it does not hold that $N\mathrm{diag}(v)B^t $ has rank $2$ for all $v\in \R^4_{>0}$. 
However, by Theorem \ref{thm:main_b}(4), only $v\in \ker(N)\cap \R^4_{>0}$ needs to satisfy this condition. We find that
\[ \mathrm{ker}(N)\cap \mathbb{R}^r_{>0}=\langle (1,1,0,0),(1,0,1,1)\rangle \cap \mathbb{R}^r_{>0}=\{ v=(a+b,b,a,a):a,b\in \mathbb{R}_{>0}\}.\]
Evaluation into \eqref{eq:conv_par_matrix} gives
\begin{equation}\label{eq:conv_par_matrix2}
N\mathrm{diag}(v)B^t = 
  \left( \begin{array}{ccccc}   -a & a & -2a \\ -a & 0 & -2a \end {array} \right),
\end{equation}
which has rank $2$ for all $a,b\in \R_{>0}$. Therefore, Theorem \ref{thm:main_b}(4) holds and $\Vp^k=\Vnd(g_k)$.

We now proceed to study whether this system has local ACR. By \eqref{eq:conv_par_matrix2}, we readily see that removal of the second column causes the rank to drop. Hence,  Theorem \ref{thm:main_b}(3) holds, and the system has local ACR with respect to $x_2$ for all choices of $k$ whenever $\Vp^k\neq \emptyset$. 

\smallskip
This system has been chosen for illustration purposes, as it is small enough to compute its solutions. The set 
 $\Vp^k$ is empty when $k_2\geq k_1$ and is described by the relations
 \[ x_1= \tfrac{k_3}{k_4x_3^2}, \qquad x_2 = \tfrac{k_1-k_2}{k_4} \]
when $k_1>k_2$. 
In this case there is ACR. The set $\Vp^k$ is a curve in $x_1,x_3$ that lives in a hyperplane given by $x_2$ constant. 
\end{Ex}

 \smallskip
The converse of Proposition~\ref{prop:checkingrank}(iii) is not true. With the notation introduced in the proof, it could  be the case that $\left( \Delta _{n-d+1}^i,\ldots ,\Delta _{n}^i\right) w_i=0$ and  $\left( \Delta _{n-d+1}^i,\ldots ,\Delta _{n}^i\right)\neq 0$. For example, if the $i$-th column of $W$ is zero,  then $h_i$ divides the polynomial $p_v(h)$, without necessarily having zero sensitivity. This is illustrated in the next example.
 
\smallskip
 
\begin{Ex}\label{ex:drop_rank} 
Consider the system $g_k(x)=k_1x_1x_2-k_2$ with $x\in \R^2_{>0}$ with matrices 
\begin{equation*}
N=\left( \begin{array}{cc}
1 & -1  
\end{array} \right) \mbox{,\; }
B=\left( \begin{array}{cc}
1 & 0  \\
1 & 0 
\end{array} \right).
\end{equation*}
Let $S=\langle (1,0) \rangle$, such that $W=(0 \ 1)$. The elements of the kernel of $N$ are parametrized as 
$(a,a)$.
The polynomial $p_v(h)$ from Proposition~\ref{prop:checkingrank}(iii) becomes
\[ \det\begin{pmatrix}
h_1a  & h_1a \\ 0 & 1 \end{pmatrix} =  h_1a. \]
This polynomial is a multiple of $h_1$. However, 
using Theorem~\ref{thm:main_b}, the minors of the matrix $N\mathrm{diag}(v)B^t = (a\ a)$ not involving column $1$ are not identically zero. Hence, the system does not have zero sensitivity nor local ACR in the variable $x_1$. 

In this case we can easily verify that the system does not have zero sensitivity in $x_1$, nor local ACR, as 
the positive solutions to $g_k(x)=0$ are of the form $x_1=\tfrac{k_2}{k_1x_2}$.
\end{Ex}

\medskip

\section{Local ACR in the context of algebraic geometry }\label{sec:backgroundAG}
The main goal of this section is to prove Theorem~\ref{prop:6regb_4} . We start by considering the case where  $g=(g_1,\dots,g_s)$ is algebraic, that is, each entry is a  polynomial. Afterwards, we will use the construction in \eqref{eq:corresp_x_z} to study generalized polynomials.

By a positive point of $\C^n$, we refer to a real point where all coordinates are positive. 
Assume $g=(g_1,\dots,g_s)$ is polynomial, and consider the  ideal associated with $g$:
\[ \mathcal{I}=\langle g_{1},\ldots ,g_{s}\rangle \subseteq \mathbb{R}[x_1,\ldots ,x_n].\]
Let $\VR (\mathcal{I} )$ and $\VC(\mathcal{I} )$ be the real  and complex algebraic varieties defined by $\mathcal{I}$. Then, as $\Omega=\R^n_{>0}$,  we have 
\[\mathcal{V}=\VR (\mathcal{I}) \cap \Rp ^n = \VC (\mathcal{I}) \cap \Rp ^n \mbox{.}\]
It is not an algebraic variety itself, but a subset of one (and a semi-algebraic variety), and we consider the induced Zariski topology.  The Zariski topology is the usual topology in algebraic geometry, having as closed sets the algebraic varieties. For details about this topology, we refer for instance to \cite{Kemper}.

Local ACR with respect to $x_i$ requires that $\Vp \subseteq \bigcup_{j=1}^\ell \VR (\langle x_i-C_j \rangle )$ for certain $C_1,\dots,C_\ell\in \Rp$. 
On the algebraic counterpart, the existence of $C_1,\dots,C_\ell\in \Rp$ such that $\prod_{j=1}^\ell x_i-C_j\in \IS $ for all $j$ is sufficient for local ACR with respect to $x_i$. However, this is not a necessary condition: for example, the existence of a univariate polynomial $f(x_i)\in \IS $ with positive roots  is also a sufficient condition (see \cite[Lemma 6.3.1, Lemma 6.5.2]{P-M}).

\subsection{Irreducible components and non-degenerate points}\label{subsec:irred_deg}  
In order to deal with this discrepancy, we need to consider a key result from Real Algebraic Geometry. To this end, we  review some results on algebraic varieties, and in particular on singular points. 
See \cite[Sections 4.6, 4.8, 9.6]{CLO} for more details and references to proofs.

The variety $\mathcal{X}=\VC (\IS)$ is irreducible if it cannot be written as the union of two proper algebraic subvarieties. This implies $\IS$ can be chosen to be a prime ideal. In general, there is a unique (up to reordering) minimal decomposition of $\mathcal{X}$ into irreducible varieties $\mathcal{X}=\mathcal{Y}_1\cup \ldots \cup \mathcal{Y}_m$,   called \emph{irreducible components}   \cite[Section 4.6, Thm 4]{CLO}.

Let $d_i\geq n-s$ be the dimension of $\mathcal{Y}_i$ for each $i=1,\ldots ,m$, such that $\mathrm{max}\{ d_1,\ldots ,d_m\} $ is  the dimension of $\mathcal{X}$. For each $x^*\in \mathcal{X}$, we consider the local dimension of $\mathcal{X}$ at $x^*$ given as $\mathrm{dim}_{x^*}(\mathcal{X})=\mathrm{max}\{ \mathrm{dim}(\mathcal{Y}_j) \mid x^*\in \mathcal{Y}_j\} $. 
By \cite[Ch II, \S1.4]{Shaf} and \cite[\S I.7, Corollary 3]{Mum_Red}, 
\begin{equation}\label{eq:sing}
n-s\leq \mathrm{dim}_{x^*}(\mathcal{X})\leq \mathrm{dim}\big(\mathrm{ker}(\tfrac{\partial g(x^*)}{\partial x})\big)
\end{equation}
for all $x^*\in \mathcal{X}$.

\begin{Def}\label{def:sing} Let $\mathbb{I}(\mathcal{X})=\langle h_1,\ldots ,h_{s'}\rangle $ be the  radical ideal defining $\mathcal{X}$. A point $x^*\in \mathcal{X}$ is 
\emph{singular} if 
\[ \rank \big(\tfrac{\partial h(x^*)}{\partial x}\big) < n-\mathrm{dim}_{x^*}(\mathcal{X})\qquad (\ \textrm{equivalenty, }\dim(\ker\big(\tfrac{\partial h(x^*)}{\partial x}\big)) > \mathrm{dim}_{x^*}(\mathcal{X}) \ ).\]
If there is an equality, then the point is \emph{non-singular}. 
\end{Def}

We denote by $\Sing (\mathcal{X})$ the set of singular points of $\mathcal{X}$.
The following proposition gathers well-known facts on $\Sing (\mathcal{X})$, as given in \cite[Section 9.6, Thm 8 and 9]{CLO}. See  \cite[Ch II, \S1.4]{Shaf}.

\begin{Prop}\label{prop:maxdim}
Let $g=(g_1,\dots,g_s)$ be a polynomial function in $\R^n$, $\IS$ the associated ideal and $\mathcal{X}=\mathbb{V}_\C(\IS)$.
\begin{enumerate}[(i)]
\item $\Sing (\mathcal{X})$ is either empty or a proper subvariety of $\mathcal{X}$, and contains no irreducible component. 
\item Points in the intersection of two irreducible components are singular. 
\item If $x^*\in\mathcal{X}$ is such that $s =\rank  \big( \tfrac{\partial g(x^*)}{\partial x}\big)$,
then $x^*$ is non-singular, and further lies in a unique irreducible component of dimension $n-s$. 
\end{enumerate}
\end{Prop} 

It follows from Proposition~\ref{prop:maxdim}(ii)  that a non-singular point $x^*\in \mathcal{X}$ belongs to a unique irreducible component. The definition of singular point requires the radical of $\mathcal{I}=\langle g_1,\dots,g_s\rangle$, which is not always easy to find. 
Proposition~\ref{prop:maxdim}(iii) allows to bypass the need of finding the radical in order to verify that a point is non-singular. The condition in this case is simply that the point is non-degenerate for the system $g(x)=0$. 
This gives rise to the following proposition, where we denote by $\Deg_\C (g)$ and $\Deg _{S,\C}(g)$   the set of degenerate complex points of the system $g(x)=0$ (resp. with respect to the extension  of a real vector subspace $S$ to $\C^n$).

\begin{Prop}\label{prop:degVSsing}
Let $\mathcal{X}=\VC (\IS)$ be an algebraic variety with $\IS=\langle g_1,\ldots ,g_s\rangle \subseteq \mathbb{R}[x_1,\ldots ,x_n]$, 
and let $S\subseteq \mathbb{R}^n$ be a vector subspace of dimension $s$. 
There is an inclusion of subvarieties of $\mathcal{X}$
\[ \Sing (\mathcal{X}) \subseteq \Deg_\C (g)\subseteq \Deg _{S,\C}(g)\subseteq \mathcal{X}. \]
In particular, if  $x^*\in \mathcal{X}\setminus \Deg_\C(g)$, then $x^*$ is non-singular and  $\mathrm{dim}_{x^*}(\mathcal{X})=n-s$.
Furthermore, if $\mathcal{I}$ is radical, then the first inclusion is an equality.
\end{Prop}

\begin{proof} Proposition~\ref{prop:maxdim}(i) gives that $\Sing (\mathcal{X})$ is an algebraic subvariety. 
The sets $\Deg_\C (g)$ and $\Deg _{S,\C}(g)$ are algebraic subvarieties of $\mathcal{X}$, as they are described by the vanishing of a set of polynomials that includes $g_1,\dots,g_s$. 

The second inclusion follows from the definition, see \eqref{eq:inclusion_deg}. For the first inclusion and $x^*\in \mathcal{X}$, by definition 
we have $x^* \notin \Deg_\C(g)$ if and only if $\dim(\mathrm{ker}(\tfrac{\partial g(x^*)}{\partial x}))=n-s$, and the latter implies $x^*\notin\Sing (\mathcal{X})$ and  $\mathrm{dim}_{x^*}(\mathcal{X})=n-s$ by Proposition~\ref{prop:maxdim}(iii). 
If $\mathcal{I}$ is radical, then $\mathcal{I} = \mathbb{I}(\mathcal{X})$, and   $\Deg_\C(g) =  \Sing (\mathcal{X})$ by Definition~\ref{def:sing} and \eqref{eq:sing}. 
\end{proof}

\begin{Rem}\label{rem:Zar_dense}
If $\mathcal{X}$ is the smallest algebraic variety containing a set $\mathcal{U}\subset \mathbb{C}^n$ (that is, $\mathcal{U}$ is Zariski dense in $\mathcal{X}$), then any polynomial that vanishes at all points of $\mathcal{U}$ also vanishes at all points of $\mathcal{X}$ (see \cite{Kemper}).
\end{Rem}

The following theorem is the key to understand the discrepancy between the geometric and algebraic interpretations of (local) ACR.

\begin{Thm}\label{thm:prime} Let $\mathcal{P}=\langle g_1,\ldots, g_s\rangle \subseteq \mathbb{R}[x_1,\ldots ,x_n]$ be a prime ideal and $\mathcal{X}=\VC (\mathcal{P})$.  If $\mathcal{X}$ has a real non-singular point $x^*$, then any polynomial $q\in \mathbb{R}[x_1,\ldots ,x_n]$  vanishing at all points of $\mathcal{X}_\R$ vanishes also at all points of $\mathcal{X}$ (that is, $q\in \mathcal{P}$).  Moreover, if for an open set $U\subseteq \R^n$, 
$x^*\in U$ and $q$ vanishes at all points of $\mathcal{X}_\R \cap U$, it also vanishes at all points of $\mathcal{X}$.
\end{Thm}
\begin{proof}
The first part follows from \cite[Prop 3.3.16]{BCR}, see also \cite[Thm 5.1]{Sot}. The second statement follows from the first part together with the fact that  $\mathcal{X}_\R \cap U\neq \emptyset$ is Zariski dense in $\mathcal{X}_\R$. 
\end{proof}

The proof of Theorem~\ref{thm:prime} uses that the intersection of a Zariski dense subset of an affine variety $\mathcal{X}\subseteq \mathbb{C}^n$ and an Euclidean open subset of $\mathbb{C}^n$ is an Euclidean open subset and, if non-empty, Zariski dense in $\mathcal{X}$. We will repeatedly use this fact in what follows.

By Theorem~\ref{thm:prime}, ACR over an irreducible component $\mathbb{V}(\mathcal{P})$ of $\mathcal{X}$ that contains positive non-singular points implies that $x_i-C\in \mathcal{P}$ for some $C>0$ (see also  \cite[Prop 6.5.3]{P-M}).
This is why the results on local ACR here concern the irreducible components with non-degenerate points.

 Recall the definitions  of $\Vp_\mathcal{C}$, $\Snd(g)$, $\SndS$,  $\Vnd (g)$ and $ \VndS$ from equations \eqref{eq:Vc}, \eqref{eq:Vnd} and \eqref{eq:VndS}. We consider now 
 $\Sreg$ to be the set of irreducible components of $\VC (\IS )$ that contain a non-singular positive  point and define $\Vreg:= \Vp_{\Sreg}$. 

 By Proposition \ref{prop:degVSsing} we have
\begin{equation*}\label{eq:inclusionsV}
\SndS\subseteq \Snd(g)  \subseteq   \Sreg \qquad \textrm{and}\qquad \VndS \subseteq \Vnd (g) \subseteq   \Vreg \subseteq \Vp.
\end{equation*}
If  $\Vnd (g)=\Vp$,  then   all irreducible components of $\mathbb{V}_\C(\mathcal{I})$ that meet the positive orthant have dimension $n-s$ and contain some non-singular (real) point (c.f. Proposition~\ref{prop:degVSsing}).
 Theorem~\ref{thm:prime} gives the following proposition on local ACR.

\begin{Prop}\label{prop:alg_lACR}
Let $g=(g_1,\dots,g_s)$ be a polynomial function in $\R^n$, and let $\mathcal{C}\subseteq \Sreg$ be a non-empty subset of irreducible components containing a positive non-singular point.
Let $\mathcal{P}_1,\ldots ,\mathcal{P}_\ell$ be the prime ideals defining the irreducible components in $\mathcal{C}$. 

The system $g(x)=0$ has local ACR  with respect to $x_i$ over $\Vp_{\mathcal{C}}$ if and only if 
there exist $C_1,\dots,C_\ell>0$ such that $x_i-C_j \in \mathcal{P}_j$ for all $j=1,\dots,\ell$.
\end{Prop}
\begin{proof}
The reverse implication is direct, as every point $x^*\in \Vp_{\mathcal{C}}$ belongs to 
(at least) one   $\mathbb{V}(\mathcal{P}_j)$. 
For the forward implication, let $C_1,\dots,C_{\ell'}>0$ be the possible different values of $x_i$ for $x\in \Vp_{\mathcal{C}}$. By hypothesis $\Vp_{\mathcal{C}} \subseteq \sqcup_{j=1}^{\ell'} \mathbb{V}(\langle x_i-C_j\rangle)$. 
Fix $j\in \{1,\dots,\ell\}$ and choose 
a non-singular point  $x^*\in \mathbb{V}(\mathcal{P}_j)\cap \R^n_{>0}$.   
As the set of singular points of $\mathbb{V}(\langle g_1,\dots,g_s\rangle)$ is Zariski closed, there exists a connected  (Euclidian) open set $U\subseteq
\mathbb{V}(\mathcal{P}_j)\cap \R^n_{>0}$ containing $x^*$ and consisting only of non-singular points.  Hence   $U\subseteq \mathbb{V}(\langle x_i-C_m\rangle)$ for some $m\in \{1,\dots,\ell'\}$. 
By Theorem~\ref{thm:prime}, $\mathbb{V}(\mathcal{P}_j)\subseteq \mathbb{V}(\langle x_i-C_m\rangle)$ and hence, as $\mathcal{P}_j$ and $\langle x_i-C_m\rangle$  are prime, $\langle x_i-C_m\rangle\subseteq \mathcal{P}_j$. This concludes the proof. 
\end{proof} 

 For $\mathcal{I}$ prime, $\mathbb{V}_\C(\mathcal{I})$ is irreducible and Proposition \ref{prop:alg_lACR} is a rephrasing of \cite[Prop 6.5.3]{P-M}, and local ACR is equivalent to ACR. 
A direct consequence of Proposition~\ref{prop:alg_lACR}, gives the following equivalence between ACR and local ACR on irreducible components (which is Proposition~\ref{prop:alg_lACR_4} for polynomials). 

\begin{Cor}\label{cor:alg_lACR}
Let $g=(g_1,\dots,g_s)$ be a polynomial function in $\R^n$, and let $\mathcal{Y}\subseteq \Sreg$ be an  irreducible component containing a positive non-singular point.
 
The system  $g(x)=0$ has local ACR  with respect to $x_i$ over $\mathcal{Y}$ if and only if it has ACR 
  with respect to $x_i$ over $\mathcal{Y}$.
 \end{Cor}

The condition $x_i-C_j \in \mathcal{P}_j$ for all $j=1,\dots,\ell$ is equivalent to $\prod_{j=1}^\ell 
(x_i-C_j) \in \bigcap_{j=1}^\ell \mathcal{P}_j$.

\begin{Rem}\label{rem:Vclosures}
In algebraic-geometric terms, Proposition~\ref{prop:degVSsing}  gives that 
the set of irreducible components of $\VC (\IS )$ that have regular points is precisely the Zariski closure of the Zariski open set $ \VC (\IS )\setminus \mathrm{Sing}( \mathcal{X})$. Intersecting further with $\R^n_{>0}$, we obtain $\Vreg$. 
Similarly, by letting the overline denote the Zariski closure, in view of Proposition~\ref{prop:degVSsing} we have:  
\begin{align*}
\Vreg & =\overline{\left( \VC (\IS )\setminus \mathrm{Sing}( \mathcal{X}))\right)\cap \mathbb{R}^n_{>0} } \cap \mathbb{R}^n_{>0} \mbox{,}\\
\Vnd (g) & =\overline{\left( \VC (\IS )\setminus \Deg (g)\right)\cap \mathbb{R}^n_{>0} } \cap \mathbb{R}^n_{>0} \mbox{,}\\
\VndS & =\overline{\left( \VC (\IS )\setminus \mathrm{Deg}_S(g)\right)\cap \mathbb{R}^n_{>0} } \cap \mathbb{R}^n_{>0} \mbox{.}
\end{align*}  
In particular,  $\Vreg\setminus  \mathrm{Sing}( \mathcal{X})$,  $\Vnd (g)\setminus \Deg(g)$ and $\VndS (g)\setminus \Deg_S(g)$
are dense Zariski  open sets of $\Vreg$, $\Vnd (g)$ and  $\VndS$ respectively.  
 \end{Rem}

\begin{Rem}
For a generic vector subspace $S$ of dimension $s$, we  have $\VndS = \Vnd (g)$. To see this, 
choose a positive non-singular point $x(i)^*$ for each irreducible component $\mathcal{Y}_i \in \Snd(g)$. Consider the $(n-s)$-dimensional tangent space $T_i$ at  $x(i)^*$. Then, for any vector subspace $S$ transversal to all $T_i$, we have $\VndS = \Vnd (g)$. 
\end{Rem}

\begin{Rem}\label{rem:ex_Deg}
It might be the case that $\Vreg(g)=\emptyset$ or $\Vnd(g)=\emptyset$, even if $\VC (\IS )$ has  a non-degenerate complex solution to $g(x)=0$. 
This occurs in Example \ref{ex:complex}, as $\VC (\IS )=\mathcal{X}_1\cup \mathcal{X}_2$ is an affine variety of dimension $1$, where $\mathcal{X}_1=\VC (\langle x_3-1,(x_1^2+x_2^2-1)^2\rangle )$ and $\mathcal{X}_2=\VC (\langle x_1^2-x_2,(x_1^2+x_2^2-1)^2+(x_3-1)^2\rangle )$.
The only intersection of $\mathcal{X}_2$ and $\R^3$ is contained in $\mathcal{X}_1$ and singular. The Jacobian matrix of $g_1,g_2$ vanishes at all points of $\mathcal{X}_1$, so all real points are   degenerate.
In this case, $\Vnd(g) = \emptyset \subsetneq \Vp $. 
This phenomenon can occur even if all irreducible components of $\VC (\IS )$ have some non-degenerate complex solution: for $g(x)=(x_1^2+x_2^2-1)^2+(x_3-1)^2$, the Jacobian of $g$ does not vanish at all points of $\VC(\mathcal{I})$,  but vanishes at all real points.
\end{Rem}

\subsection{A criterion for local ACR over $\Vreg$.}\label{subsec:criterion_lACR}

In this section we prove Theorem~\ref{prop:6regb_4} and provide a criterion to check Theorem \ref{thm:6_5}(3) when $g$ is polynomial.

\begin{Prop}\label{prop:6reg} 
Let $g=(g_1,\dots,g_s)$ be a polynomial function in $\R^n$, $\mathcal{X}= \mathbb{V}_\C(\langle g_1,\dots,g_s \rangle)$ and write the radical ideal defining $\mathcal{X}$ as $\mathbb{I}(\mathcal{X})=\langle h_1,\ldots ,h_{s'}\rangle$.
Consider  $\mathcal{Y}\in \Sreg$ such that $\mathcal{Y}\cap \R^n_{>0}\neq \emptyset$.  

System $g(x)=0$ has (local) ACR with respect to $x_i$ over $\mathcal{Y}\cap \R^n_{>0}$ if and only if 
\begin{equation}\label{eq:rank_dim}
 \rank\Big( \big( \tfrac{\partial h(x^*)}{\partial x}\big)^i \Big)<  n-\mathrm{dim}(\mathcal{Y})
 \end{equation} for all $x^*\in \mathcal{Y}\cap \R^n_{>0}$.  
\end{Prop}

\begin{proof}
Write  $\mathcal{Y}=\VC (\mathcal{P})$ with $\mathcal{P}$ prime. As $\mathcal{Y}$ contains a non-singular positive point,  the system has ACR with respect to $x_i$ over $\mathcal{Y}$ if and only if  $x_i-C \in \mathcal{P}$ for some  $C\in \Rp $ by Proposition \ref{prop:alg_lACR}.
Hence all we need is to show that the    latter occurs if and only if \eqref{eq:rank_dim} holds for all $x^*\in \mathcal{Y}\cap \R^n_{>0}$. As $\mathcal{Y}\cap \R^n_{>0}\neq \emptyset$, the latter is equivalent to require that \eqref{eq:rank_dim}  holds  for all $x^*\in \mathcal{Y}$.

Let  $m=n-\mathrm{dim}(\mathcal{Y})$, write $\mathcal{P}=\langle f_1,\ldots ,f_r\rangle $ and $f=(f_1,\dots,f_r)$.
For the forward implication, assume that $x_i-C \in \mathcal{P}$ for some $C\in \Rp $. Then $\frac{\partial (x_i-C)}{\partial x}=e_i$ (the canonical vector in $\R^n$) and
\begin{equation}\label{eq:thm6reg}
e_i\in \mathrm{rowspan}\left( \tfrac{\partial f(x^*)}{\partial x}\right)\qquad \textrm{for all } \, x^*\in \mathcal{Y}.
\end{equation}
Equivalently, the tangent space to $\mathcal{Y}$ at any point is perpendicular to $e_i$. 
If $x^*\in \mathcal{Y}$ is not singular as a point in $\VC (\IS)$, then the tangent spaces to $\mathcal{Y}$ and to $\VC (\IS )$ at $x^*$ agree, and further $\rank \big(\frac{\partial h(x^*)}{\partial x}\big)=m$. Hence \eqref{eq:thm6reg} implies $e_i \in  \mathrm{rowspan}\big(  \frac{\partial h(x^*)}{\partial x}\big)$ for all $x^*\in \mathcal{Y}$ non-singular in  $\VC (\IS)$, or equivalently 
\begin{equation}\label{eq:rank}
\mathrm{rank}\Big( \big(\tfrac{\partial h(x^*)}{\partial x}\big)^i\Big) <m
\end{equation} 
for all  $x^*\in \mathcal{Y}$  non-singular in  $\VC (\IS)$. 
Condition \eqref{eq:rank} is polynomial, and the set of points in $\mathcal{Y}$ that are non-singular in $\VC (\IS)$ is Zariski dense in $\mathcal{Y}$. Hence 
\eqref{eq:rank} holds for all   $x^*\in \mathcal{Y}$ non-singular in $\VC (\IS)$ if and only if it holds for all $x^*\in \mathcal{Y}$. 
\medskip

We consider now the reverse implication and assume that \eqref{eq:rank} holds for all non-singular points  in $\mathcal{Y}$. This means that $e_i$ is orthogonal to any vector tangent to $\mathcal{Y}$ at a non-singular point. 
Fix a non-singular point $x^*$  in $\mathcal{Y}$, and let $U\subseteq \R^n$ be an open set with the Euclidian topology, such that $\mathcal{Y}\cap U$ is a connected and consists of non-singular points. Then $\mathcal{Y}\cap U$ is a real differentiable manifold, and furthermore, is Zariski dense in $\mathcal{Y}$. We show that there exists a polynomial of the form $x_i-C$
that vanishes on $\mathcal{Y}\cap U$, and this  implies that it also vanishes on $\mathcal{Y}$ as desired.

Given another point $y^*$ in  $\mathcal{Y}\cap U$, let $\alpha\colon [0,1]\rightarrow \mathcal{Y}\cap U$ be a differentiable curve connecting $x^*$ to $y^*$. 
Then, for any $t$,  $\alpha'(t)$ is orthogonal to $e_i$ by assumption, as it is tangent to a point of $\mathcal{Y}$. It follows that $\alpha'(t)_i=0$ for $t\in [0,1]$. 
Let $\pi\colon U \rightarrow \R$ be the projection onto the $i$-th component. The gradient $\tfrac{\partial \pi}{\partial x}$ of $\pi$
is zero everywhere but $1$ in the $i$-th component. 
Hence $\tfrac{\partial \pi(\alpha(t))}{\partial x}\cdot \alpha'(t)=0$ for all $t\in [0,1]$.
By the gradient theorem, we have
\[ 0 = \int_0^1\tfrac{\partial \pi(\alpha(t))}{\partial x}\cdot \alpha'(t) \, dt = \pi(\alpha(1))-\pi(\alpha(0)) = y^*_i - x^*_i. \]
This shows that $x_i$ is constant on $\mathcal{Y}\cap U$, and concludes the proof of the proposition. 
\end{proof}

Proposition~\ref{prop:6reg} can be used to  identify local ACR in components with non-singular points, but requires the use of the radical of $\langle g_1,\dots,g_s\rangle$. In applications, where $g$ has parametric coefficients, the radical ideal might not be easy to find. This problem is overcome when considering components with non-degenerate points of $g(x)=0$. 

\smallskip
We are now ready to prove Theorem~\ref{prop:6regb_4}.

\begin{proof}[Proof of Theorem~\ref{prop:6regb_4}]
Assume first that $g$ is polynomial.
Let $\langle h_1,\ldots ,h_{s'}\rangle$ be the radical of $\mathcal{I}=\langle g_1,\dots,g_s\rangle$. 
By Proposition~\ref{prop:degVSsing}, any irreducible component $\mathcal{Y}\in \Snd(g)$ has dimension $n-s$ and contains non-singular points. 
It follows that for any $x^*\in  \Vp_{\mathcal{C}}$ that is non-singular in $\mathbb{V}_\C(\mathcal{I})$, we have
\[ s = \rank\Big( \big( \tfrac{\partial h(x^*)}{\partial x}\big)\Big)=  \rank\Big( \big( \tfrac{\partial g(x^*)}{\partial x}\big)\Big), \quad \textrm{and hence} \quad \rank\Big( \big( \tfrac{\partial g(x^*)}{\partial x}\big)^i \Big) =  \rank\Big( \big( \tfrac{\partial h(x^*)}{\partial x}\big)^i \Big). \]
Now the claim follows from Proposition~\ref{prop:6reg}.

If $g$ is not polynomial, we consider $\widetilde{g}$ from \eqref{eq:gtilde}. System $g(x)=0$ has local ACR with respect to $x_i$ over  $\Vp_{\mathcal{C}}$ if and only if 
 the polynomial system $\widetilde{g}(z)=0$ has local ACR with respect to $z_i$ over $\varphi^{-1}(\Vp_{\mathcal{C}})$, with $\varphi$ in \eqref{eq:corresp_x_z}. The claim now follows using that condition \eqref{eq:rank_dim2_4}  on $\widetilde{g}$, which characterizes  local ACR for $\widetilde{g}$,  is equivalent to condition \eqref{eq:rank_dim2_4}  on $g$ by \eqref{eq:jacg}.

\end{proof}

We conclude this section with a criterion to check Theorem \ref{thm:6_5}(3) when $g$ is polynomial.

\begin{Prop}\label{thm:6} 
Let $g=(g_1,\dots,g_s)$ be a polynomial function in $\R^n$  and $\mathcal{C}$ a non-empty set of irreducible components of $\Vp$. 
The following are equivalent:
\begin{enumerate}
 	\item 
	$\rank\Big( \big( \tfrac{\partial g(x^*)}{\partial x}\big)^i \Big)<  s$ for all $x^*\in \Vp_\mathcal{C}$.
	\item All ($s\times s$)-minors of $\frac{\partial g(x)}{\partial x}$ not involving column $i$ belong to $\bigcap _{\mathcal{P} \, \textrm{prime}\, \mid\, \VC (\mathcal{P})\in \mathcal{C}}\mathcal{P}$.
\end{enumerate}
\end{Prop}

\begin{proof}
Statement (2) is a rephrasing of (1) using Theorem \ref{thm:prime}, as (1) indicates that all ($s\times s$)-minors of $\frac{\partial g(x)}{\partial x}$ not involving column $i$ vanish on the positive real part of the variety defined by the ideal $\bigcap _{\mathcal{P} \, \textrm{prime}\, \mid\, \VC (\mathcal{P})\in \mathcal{C}}\mathcal{P}$. 
\end{proof}

  \section{Conclusion}  
In this work we considered generalized polynomial systems  in $\R^n$
and studied the robustness of the set of solutions in $\Rp^n$ (denoted by $\Vp$).
The robustness of the elements in $\Vp$ with respect to a variable $x_i$ can be measured at different scales, yielding to the two properties studied here in detail: zero sensitivity and (local) ACR. 

We have provided a simple criterion (Theorem \ref{thm:k_fix_rankVSsens0}) for determining zero sensitivity of a system with respect to $x_i$ at a non-degenerate point of $\Vp$ (with respect to a vector subspace). The criterion is based on the drop of the rank of the Jacobian matrix of $g$ after removing the $i$-th column. Importantly, this criterion can be stated and  implemented as a computational algebra problem.
An analogous criterion is obtained for local ACR (Theorem~\ref{prop:6regb_4}), and this allowed us to determine when local ACR and zero sensitivity agree   (Theorem \ref{thm:6_5}). 

Deciding upon ACR is a difficult problem (theoretically solvable, but not decidable for realistic systems with the currently available computational power). 
Thanks to the connection to zero sensitivity and the introduction of local ACR, we have provided a practical test to decide upon ACR. The criterion turns out to involve only linear algebra on symbolic matrices when considering families of parametrized systems arising from reaction network (Theorem \ref{thm:main_b}), and hence gives the possibility to scan databases of reaction networks for the occurrence of ACR.

\bigskip
\noindent
\textbf{Acknowledgements. } BP acknowledges funding from the European Union's Horizon 2020 research and innovation
programme under the Marie Sklodowska-Curie IF grant agreement No 794627. EF has been supported by the Independent Research Fund of Denmark and by the Novo Nordisk Foundation grant NNF18OC0052483. 

\end{document}